\documentclass[11pt]{article}
\usepackage{geometry}                
\usepackage{graphicx}
\usepackage{amsmath}
\usepackage{amsthm}
\usepackage{amssymb}
\usepackage{leftidx}
\usepackage{epstopdf}
\usepackage{comment}
\usepackage{verbatim}

\newtheorem{theorem}{Theorem}
\newtheorem{lemma}{Lemma}
\newtheorem{fact}{Fact}
\newtheorem{corollary}{Corollary}

\theoremstyle{definition}

\newtheorem{definition}{Definition}

\newtheorem{remark}{Remark}
\newtheorem{question}{Question}

\begin{document}

\newcommand {\dd}{\textup{d}}
\newcommand{\RR}{\mathbb{R}}
\newcommand{\slipparameter}{\beta}
\newcommand{\eigparam}{\lambda}
\newcommand{\remaind}{{\rm rem}}
\newcommand{\steady}{{\rm steady}}
\newcommand{\boxshape}{{box}}
\newcommand{\torsforce}{{\nu}}

\newcommand{\CM}{{\cal CM}}
\newcommand{\AM}{{{\mathbf A}\mathbf M}}
\newcommand{\CBF}{{\cal CBF}}

\newcommand{\citeKWRob}{ Part II }
\newcommand{\citeKeconv}{ Part I }




\title{Inequalities for the fundamental Robin eigenvalue of the Laplacian for box-shaped domains}
\author{Grant Keady and Benchawan Wiwatanapataphee\\
Department of Mathematics and Statistics\\
Curtin University, Bentley, 6102\\
Western Australia}
\date{\today}                                           

\maketitle
\vspace{0.3cm}
\subsection*{Abstract}
This document consists of two papers, both submitted, and supplementary material.
The submitted papers are here given as Parts I and II.
Their abstracts are given at pages~\pageref{pg:absI} and~\pageref{pg:absII} respectively.
The supplementary material is given in appendices.

\bigskip
\noindent
The supplementary material includes material directed at answering questions
stated in, or implicit in, the submitted papers.
The maple code referenced in the supplements is available via\\
\verb$https://sites.google.com/site/keadyperthunis/home/papers$\\
One of the questions is as follows.
Let $\phi_1(z)=\arctan(1/z)/z$. ($\phi_1$ is, itself, completely monotone.)
Is the inverse of $\phi_1$, denoted $\mu$ in the following, also completely monotone?
Various approaches have been considered.
Direct calculation of the higher derivatives of $\mu$ --
with maple to handle the lengthy expressions -- and an inductive argument is
considered in work presented at page~\pageref{pg:pghigher}.
Some complex variable approaches are instigated at page~\pageref{pg:pgcomplex}.

\newpage

\begin{center}
{\huge{\textbf{\textsc{ Part I:
On functions and inverses, \\
both positive, decreasing and convex
}}}}
\end{center}

\bigskip
\section*{Abstract}
\label{pg:absI}
Any function $f$ from $(0,\infty)$ onto $(0,\infty)$
which is decreasing and convex
has an inverse $g$ which is positive and decreasing -- and convex.
When $f$ has  some form of generalized convexity we determine
additional convexity properties inherited by $g$.
When $f$ is positive, decreasing and $(p,q)$-convex, its inverse $g$ is $(q,p)$-convex.
Related properties which pertain when $f$ is a Stieltjes function are developed.
The results are illustrated with the Stieltjes function $f(x)=\arctan(1/\sqrt{x})/\sqrt{x}$, a function which
arises, via a transcendental equation,  in an application presented in Part II.
\bigskip

\section{Introduction}\label{sec:Intro}

If $f$ is positive and (strictly) decreasing  then, clearly, it has an inverse which is positive and decreasing:
the inverse will be denote by $g$.
If, in addition, $f$ is convex, it is easy to show that $g$ is also convex.
For ease of exposition, assume appropriate differentiability.
Denote derivatives with a prime.
Starting from $f(x)=y$, $g(y)=x$  and $f(g(y))=y$, and thence $f' g'=1$, a further differentiation gives
\begin{equation}
 f'' (g')^2 + f' g'' = 0 . 
 \label{eq:fgconv}
 \end{equation}
Using $f'<0$ it follows that, if one of $f$ or $g$ is convex, so is the other.
This result is well-known, as is the fact that one can remove the differentiability assumptions.
See~\cite{Mr08} and Proposition 1 of~\cite{HUML03}.

Elementary items like this will labelled as `Facts', reserving the word `Theorem'  for more
significant ones or those which the author has not found elsewhere.\\
\begin{fact}\label{fac:Fact1}
{The inverse of a positive, decreasing convex function is positive, decreasing and convex.}
\end{fact}
\medskip

Another elementary fact is\\
\begin{fact}\label{fac:Fact2}
{Both the sum, $f_0+f_1$, and the product, $f_0 f_1$, of a pair of positive, decreasing convex functions, $f_0$ and $f_1$, are positive, decreasing and convex.}
\end{fact}

The preceding facts are well-known but straightforward developments, 
in which further convexity properties of $f$ yield further properties of its inverse,
do not appear to have had any systematic treatment.
In this connection a small development of equation~(\ref{eq:fgconv}) is useful:
\begin{equation}
 \frac{g''}{(g')^2} =  - \, \frac{f''}{f'}  \qquad{\rm so} \quad
 \frac{gg''}{(g')^2} =  - \, \frac{ x f''}{f'} . 
\label{eq:powgen}
\end{equation}

The background to this paper is that an application -- studied in~\citeKWRob -- required
knowledge of convexity or concavity properties of a function $\mu$ defined
through an implicit equation.
It happened that the inverse of this function, denoted $\phi_1$, and the closely related function $\phi_2$,
\begin{equation}
\phi_1(z)=\frac{1}{z} \, \arctan\left( \frac{1}{z} \right) ,\qquad
\phi_2(z)=\phi_1(\sqrt{z}) ,
\label{eq:phiarctan}
\end{equation}
happen to be completely monotone and have various other easily established properties.
The function $\phi_2$ is Stieltjes.
(For definitions, see~\S\ref{sec:CMStieltjes}.)
While calculations with the implicitly defined function were possible --
and, indeed were made before the vastly neater approach here was discovered -- 
the routine calculation effort was somewhat unpleasant, and, perhaps, sometimes ugly.
Avoiding such ugly calculation
led to the topic of this paper: determining convexity properties of the inverse of a function from
those of the original function.

Despite extensive searching, no systematic treatment of the topic of this paper has been found elsewhere.
Very simple and elegant results -- our Theorem 1 for example -- were waiting to be discovered.
We have attempted to report any previously published item concerning inverses related to ours.

There are two directions in this paper.\\
1. The first is to establish properties of $\phi_1$, $\phi_2$ and their inverses such as
are needed to establish the results in~\citeKWRob.
This goal is achieved.\\
2. Establishing properties of $\phi_1$ and $\phi_2$ proved to be easiest by beginning
by noting that they are completely monotone, and $\phi_2$ is Stieltjes.
Very early on in this study we obtained evidence, based on the first few derivatives of the 
inverses of the $\phi$, that the following questions might be answered affirmatively.
\begin{question}\label{que:q1}
Is the inverse of $\phi_2$ completely monotone? \\
Is the inverse of $\phi_1$ completely monotone (and, perhaps even Stieltjes)?
\end{question}
These questions  -- and the more general ones in~\S\ref{sec:CMStieltjes} -- remain unresolved.
We have attempted to summarise general facts about the inverses of Stieltjes functions
as concisely as possible: see~~\S\ref{sec:CMStieltjes}.
Some of the further material on $\phi_1$ and $\phi_2$ later in the paper has, as yet, not been used.
However, it may be useful to others in attempting to answer Questions~\ref{que:q1}.

\medskip

An outline of the paper follows.\\
In \S\ref{sec:secFirst}, by way of introduction to the general fact
that derivatives of a function $f$ allow one to find properties of its inverse, $g$,
we consider first derivatives.
A memorable result is that, if $f$ is a Stieltjes function not equal to a multiple of $1/x$, the inverse of $f$
is not a Stieltjes function.\\
%
In  \S\ref{sec:secSecond} we treat second derivatives and convexity matters.
The main result from this section is Theorem~\ref{thm:thmpqqp} which establishes that
if a positive, decreasing, convex function $f$ is $(p,q)$-convex, its
inverse $g$ is $(q,p)$-convex.
Various special cases of  `convexity with respect to means' are noted.\\
%
In~\S\ref{sec:CMStieltjes} we note convexity properties of Stieltjes functions.
The Stieltjes function $\phi_2$ is the key to neat proofs of properties of its inverse.\\
%
In \S\ref{sec:transcendental} we present 
the transcendental equation which led to this study,
and the properties of $\phi_1$, $\phi_2$ and their inverses $\mu$ and $\mu_{(2)}$.\\
%
In~\S\ref{sec:Othertr} we instance a few
other situations where a function is defined through a transcendental equation
and the methods of this paper might be useful.
This, however, is very much an aside to our main focus: properties of $\phi_2$, $\phi_1$
and their inverses.\\
%
In~\S\ref{sec:questions}
we conclude with a short discussion and comment on the questions asked earlier in the paper.

\medskip

In our application, Stieltjes functions, denoted ${\cal S}$, and/or their reciprocals, the
Complete Bernstein Functions, denoted $\CBF$ arise.
In view of this, and, possible future applications, we discuss these at various points in this paper,
notably in~\S\ref{sec:CMStieltjes} and subsequently.
As an instance of the sort of neat similarity that sometimes occurs we note the elementary fact:\\
{\it A  function $h : (0,\infty)\rightarrow{R}$ is convex (concave), if and only if $xh(1/x)$ is convex (concave)}.\\
This has a parallel, see~\cite{SSV} p61 equation (7.3):\\
{\it A function $h : (0,\infty)\rightarrow{R}$ is $\CBF$ if and only if $xh(1/x)$ is $\CBF$.}\\
{\it A function $s : (0,\infty)\rightarrow{R}$ is $\cal S$ if and only if $s(1/x)/x$ is $\cal S$.}

As an example concerning the last statement, we mention the Stieltjes function $\phi_2$
occuring in our application:
$$\phi_2(z)
:= \frac{1}{\sqrt{z}} \, \arctan\left( \frac{1}{\sqrt{z}} \right) ,\qquad{\rm and}\ \ 
\frac{1}{z}\,\phi_2(\frac{1}{z})
= \frac{\arctan(\sqrt{z})}{\sqrt{z}} . $$
The latter function is also a Stieltjes function, which is easily
established independently.
See the table in~\S\ref{subsec:subsecStieltjes}.

\section{First derivatives}\label{sec:secFirst}

This section is genuinely elementary, with extremely simple calculations.
Nevertheless it yields results which may not be very well known.

As mentioned earlier, an interest arising from our $\phi_2$ being Stieltjes is
what can be said about the inverses of Stieltjes functions.
Completely monotone and Stieltjes functions are treated in \S\ref{sec:CMStieltjes},
but some easy results can be stated now.
The first concerns first derivatives only.
\medskip

\begin{fact}\label{fac:Fact3}\ {Let $f$ be a positive decreasing (differentiable) function on $x>0$.\\
(i) Each of the functions $f(\frac{1}{x})/x$ and $1/(x f(x))$ is positive and decreasing iff
\begin{equation}
x f'(x) + f(x) >0\qquad {\rm for\ \ } x>0\  .
\label{eq:resf1oxox}
\end{equation}
(ii) The set of positive decreasing functions satisfying inequality~(\ref{eq:resf1oxox}) is closed under addition.\\
(iii) For any positive, decreasing $f$ satisfying~(\ref{eq:resf1oxox}),
its inverse, denoted by $g$, is positive, decreasing, and
\begin{equation}
 y g'(y) + g(y) < 0 \qquad {\rm for\ \ } y>0 \  .
 \label{eq:resg1oxox}
\end{equation}
Hence that each of the functions 
$g(\frac{1}{y})/y$ and $1/(y g(y))$
is positive and increasing.}
\end{fact}

\begin{proof} (i) follows from routine differentiation. (ii) is obvious.\\
(iii) Writing $x=g(y)$ and $y=f(x)$ and noting $f'(x)=1/g'(y)$,
inequality~(\ref{eq:resf1oxox}) rewrites to
$$ \frac{g(y)}{g'(y)}+y >0 .$$
Since $g'(y)<0$, inequality~(\ref{eq:resg1oxox}) follows.
The result in the final sentence is established with routine diffferentiation much as in (i).
\end{proof}

\medskip
\noindent{\bf\sc Example.}
\begin{align*}
f(x)
&= \frac{1}{\sqrt{x}}, \qquad x f'(x) + f(x)= \frac{1}{2\sqrt{x}},\\
g(y)
&= \frac{1}{y^2}, \  \qquad  y g'(y) + g(y)= - \frac{1}{y^2} .
\end{align*}

\medskip
\begin{corollary}\label{cor:corfac3} 
{The only (positive decreasing) Stieltjes functions $f$ whose inverses
are also Stieltjes functions are the positive multiples of  (the involution) $f(x)=1/x$.
}
\end{corollary}

\begin{proof} Proof. One can appeal to the results that for any Stieltjes function $f$
each of the functions
$f(\frac{1}{x})/x$ and $1/(x f(x))$ is Stieltjes.
This can also be proved directly, as follows.
Since, for $t\ge{0}$,
\begin{equation}
x \frac{d}{dx}\frac{1}{(x+t)} \ +\ \frac{1}{(x+t)}= \frac{t}{(x+t)^2}>0 ,
\label{eq:eqxpt1}
\end{equation}
the function $1/(x+t)$ satisfies inequality~(\ref{eq:resf1oxox}), and, hence,
using equation~(\ref{eq:StDef}),
so does any Stieltjes function.

Any Stieltjes function is positive and decreasing.
The previous paragraph ensures that Fact~\ref{fac:Fact3} applies and so
inequality~(\ref{eq:resf1oxox}) is satisfied

Finally the inverse $g$ of $f$ does not satisfy inequality~(\ref{eq:resf1oxox})
so it is not Stieltjes.
\end{proof}


 (As an aside here, we remark that there are many completely monotone involutions
 mapping $(0,\infty)$ to $(0,\infty)$, and, as in the previous Corollary,
 none except $1/x$ is Stieltjes.)
 
 \medskip
 As we have already noted $\phi_2$ satisfies inequality~(\ref{eq:resf1oxox}).
 The inverse of $\phi_1$, denoted $\mu$ below, also satisfies inequality~(\ref{eq:resf1oxox})
 (and, in so doing, shares this property with Stieltjes functions).
 
 \medskip
 As a lead in to consideration of higher derivatives we remark that equation~(\ref{eq:eqxpt1})
 generalises in several ways, for example
 \begin{equation}
 \frac{d^n}{d x^n}\left( \frac{x^n}{x+t}\right) 
 =\frac{ n!\, t^n}{(x+t)^{n+1} } ,
 \label{eq:eqxpt2}
 \end{equation}
 and
 \begin{equation}
\frac{d}{d x}\left(x^n \frac{d^{n-1}}{d x^{n-1}}\left( \frac{1}{x+t}\right) \right)
 =\frac{(-1)^{n-1} n!\, x^{n-1} t}{(x+t)^{n+1} } , 
 \label{eq:eqxpt3}
 \end{equation}
The $n=2$ version of the latter is the $(p=-1,q=1)$ result on a case of
$(p,q)$-concavity of Stieltjes functions.
See Definition~\ref{def:pq}.
Theorem~\ref{thm:Stieltjes1ox} is a general statement of the
HA-concavity of Stieltjes functions.

\section{Second derivatives and $(p,q)$-convexity}\label{sec:secSecond}

\begin{definition}\label{def:pq}
The function $f$ is {\it $(p, q)$-convex} ({\it $(p, q)$-concave}) if and only if 
$x \mapsto x^{1-p} f'(x)(f(x))^{q-1}$  is increasing (decreasing).
See~\cite{BY14},~~\cite{Ba10}.
\end{definition}

Special cases arise sufficiently frequently that there are other notations.
There is some literature, notably\cite{AVV07}, in connection with `convexity with respect to means',
and the  letters A for `arithmetic',
G for `geometric',
and H for `harmonic' are used to label these.
For example, AA-convex is ordinary convexity,
AG-convex means log-convex, etc.
The set of $(p,1)$-convex functions is obviously closed under addition.
$(1,q)$-convexity is related to power-convexity defined and discussed below.

\begin{center}
\begin{tabular}{|c|c|c|c|}
\hline
$(p,q)$& MN& & \\
\hline
(1,1)&    AA& $f(t)$ convex in $t$& $f\left(\frac{x_0+x_1}{2}\right)\le \frac{f(x_0)+f(x_1)}{2}$\\
        &          &  decreasing in $t$& \\
        \hline
(1,0)&    AG& $\log(f(t))$ convex in $t$&  $f\left(\frac{x_0+x_1}{2}\right)\le \sqrt{f(x_0) f(x_1)}$\\
        &          &  decreasing in $t$ & \\
        \hline
(1,-1)&   AH& $1/f(t)$ concave in $t$& $f\left(\frac{x_0+x_1}{2}\right)\le \frac{2f(x_0)\, f(x_1)}{f(x_0)+f(x_1)}$\\
         &          & increasing in $t$ & \\
         \hline
(0,1)&    GA& $f(\exp(t))$ convex in $t$& $f(\sqrt{x_0 x_1})\le \frac{f(x_0)+f(x_1)}{2}$  \\
         &          &  decreasing in $t$ & \\
         \hline
(0,0)&   GG&  $\log(f(\exp(t)))$ convex in $t$&  $f(\sqrt{x_0 x_1})\le \sqrt{f(x_0)\,f(x_1)}$ \\
         &          &  decreasing in $t$ & \\
         \hline
(0,-1) &    GH & $1/f(\exp(t))$ is concave in $t$& $f(\sqrt{x_0 x_1})\le \frac{2f(x_0)\, f(x_1)}{f(x_0)+f(x_1)}$ \\
          &          & increasing in $t$& \\
          \hline
(-1,1)&   HA& $f(1/t)$ convex in $t$&  $ f(\frac{2x_0\,x_1}{x_0+x_1})\le \frac{f(x_0)+f(x_1)}{2}$ \\
         &          & increasing in $t$ & \\
         \hline
(-1,0)&   HG&$\log(f(1/t))$ is convex in $t$  & $ f(\frac{2x_0\,x_1}{x_0+x_1})\le \sqrt{f(x_0)\,f(x_1)}$ \\
          &        & increasing in $t$ & \\
          \hline
(-1,-1)&  HH& $1/f(1/t)$ concave in $t$& $f(\frac{2x_0\,x_1}{x_0+x_1})\le \frac{2f(x_0)\, f(x_1)}{f(x_0)+f(x_1)}$   \\
         &          &   decreasing in $t$ & \\
\hline
\end{tabular}
\end{center}

Inclusions of these sets of MN-convex (MN-concave) functions are well established.
See~\cite{AVV07}.
The arrows are to be read, as for example from the entry at right:\\
$f$ AG-convex (log-convex) implies $f$ is AA-convex (ordinarily convex).\\
The vertical arrows require of the function $f$ that it be positive, decreasing.

$$
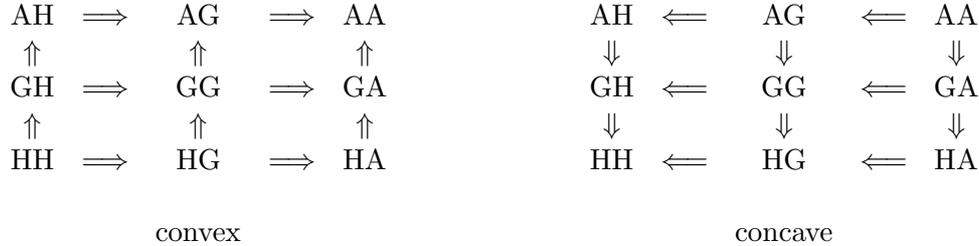
\begin{array}{ccccc}
{\rm AH} & \Longrightarrow & {\rm AG} & \Longrightarrow & {\rm AA}\\
\Uparrow& & \Uparrow & & \Uparrow\\
{\rm GH} & \Longrightarrow & {\rm GG} & \Longrightarrow & {\rm GA}\\
\Uparrow& & \Uparrow & & \Uparrow\\
{\rm HH} & \Longrightarrow & {\rm HG} & \Longrightarrow & {\rm HA}\\
& & & & \\
& &{\rm convex}& &\\ 
& & & & 
\end{array}
\qquad\qquad\qquad
\begin{array}{ccccc}
{\rm AH} & \Longleftarrow & {\rm AG} & \Longleftarrow & {\rm AA}\\
\Downarrow& & \Downarrow & & \Downarrow\\
{\rm GH} & \Longleftarrow & {\rm GG} & \Longleftarrow & {\rm GA}\\
\Downarrow& & \Downarrow & & \Downarrow\\
{\rm HH} & \Longleftarrow & {\rm HG} & \Longleftarrow & {\rm HA}\\
& & & & \\
& &{\rm concave}& &\\ 
& & & & 
\end{array}
$$

A few remarks on the cases when$MN$-convexity is closed under addition are in order.
(This is partly because this is useful for functions defined by integrals whose ntegrands 
are products of positive functions with kernels having such convexity.
We use this in~\S\ref{sec:CMStieltjes}.
An illustration of the use when a set of functions is closed under addition has
already been given in the first paragraph of the proof of Corollary~\ref{cor:corfac3}
in which Fact~\ref{fac:Fact3}(ii) is a ley ingredient.)
\begin{itemize}  
\item It is clear from the Definition~\ref{def:pq} that when $q=1$, the set of
$(p,1)$-convex (concave) functions is  closed under addition, and obviously remains so when
the functions are also positive and decreasing.
Thus the AA, GA and HA entries above form (convex) cones.

\item There are other sets which form cones, notably the positive decreasing
AHconvex functions and AG-convex functions.
This case, $p=1$ , and $q\le{1}$ is often treated in its own right
as {\it power-convex} functions: see Definition~\ref{def:defpowq} below.

\item The proof that the positive log-convex, AG-convex, functions form a cone can be adapted to
the GG-convex and HG-convex functions.
The AGM inequality can be used to establish, for positive numbers
$$\sqrt{ab}+\sqrt{cd}\le \sqrt{(a+c)(b+d)} .$$
Applying this in the form
$$\sqrt{f_0(x)f_0(y)}+\sqrt{f_1(x)f_1(y)}
\le\sqrt{(f_0+f_1)(x)\, (f_0+f_1)(y)} ,
$$
yields the results.

\item See also~\cite{NiP04} p91 Exercise 2.
\end{itemize}

\begin{definition}\label{def:defpowq}
A nonnegative function $f$ is said to be {\it $q$-th power convex} if, for $q\ne{0}$, $q (f(x)^q$ is convex, and
{\it $0$-power convex} if $\log(f)$ is convex,
also called log-convex, or  as in~\cite{AVV07}, AG-convex. See~\cite{Li82}.
\end{definition}
(When $q<0$, and $f$ is $q$-th power convex, then $f(x)^q$ is concave.)\\
{\it if $f\ge{0}$ is $q_0$-th power convex then it is $q_1$-th power convex for $q_1\ge{q_0}$.}\\
Another property, used here and again in our application in~\citeKWRob, is, from p159 of~\cite{Li82}:\\
{\it If $q\le{1}$ then the set of positive, decreasing, convex functions which are
$q$-th power convex is closed under addition.
This set is a convex cone in appropriate function spaces.}

Whether the set is closed under multiplication is less important in our present application.
As an aside we recall Fact~\ref{fac:Fact2} and mention
(noting that definitions from~\S\ref{sec:CMStieltjes} are needed for later items):
\begin{fact}\label{fac:FactProd}
{\it The product of AG-convex functions is AG-convex. The product of AG-concave functions is AG-concave.\\
The product of GG-convex functions is GG-convex. The product of GG-concave functions is GG-concave.\\
The product of two functions $f_0,~f_1\in$AH is not, in general in AH, but
$\sqrt{f_0 f_1}\in$AH, i.e.
$1/\sqrt{f_0 f_1}$ is concave.\\
The product of completely monotone functions is completely monotone.\\
The product of two Stieltjes functions $f_0,~f_1$ is not, in general Stieltjes, but
$\sqrt{f_0 f_1}$ is.}
\end{fact}
(In connection with the GG functions, see also~\cite{NiP04} Lemma 2.3.4.)

\medskip
In the proof of Theorem~\ref{thm:thmpqqp} we use the notation
\begin{equation}
D(f(x),p,q)
 :=\frac{d}{dx}\left( x^{1-p} f'(x)(f(x))^{q-1}\right)
 \label{eq:Ddef}
 \end{equation}
 Before proving the theorem we note a simple identity
 (which can be used in connection with $\mu$ and $\mu^2$ in~\S\ref{subsec:Propmu}):
 $$
 \frac{D(f(x)^2,p,q)}{D(f(x),p,2q)}=2 \\
 $$


\begin{theorem}\label{thm:thmpqqp}
If a positive, decreasing, convex function $f$ is $(p,q)$-convex, its inverse $g$ is $(q,p)$-convex.
\end{theorem}

\begin{proof} 
 $f$ is $(p,q)$-convex iff
$$  x^p f^{1-q} D(f(x),p,q)
= x f'' + x\,(q-1)\, \frac{(f')^2}{f} + (1-p)\, f' > 0 . $$
Setting $f''=-g''/(g')^2$, $f'=1/g'$, $f(x)=y$ and $x=g(y)$ in the preceding equation gives
$$-\frac{g g''}{(g')^3}+ \frac{(q-1)g}{y (g')^2} +\frac{1-p}{g'}
=    - \frac{g(y g'' + (1-q)\, g' + \frac{y(p-1)\, (g' )^2}{g})}{y (g')^3} >0 $$
Using that $g'<0$ the term in parentheses in the numerator of the long expression above
is positive. This is the result that $g$ is $(q,p)$-convex.
\end{proof}

\medskip

There are corresponding results for positive, decreasing, $(p,q)$-concave functions.
Also, when the functions are increasing rather than decreasing,
convexity of $f$ gives concavity of $g$ and vice-versa.

\medskip

{\sc Example.} $D(1/x^\alpha,p,q)=x^{-1-p-\alpha q}\alpha(p+\alpha q)$.
From this the convexity properties of the positive, decreasing convex functions
 $1/\sqrt{x}$ and its inverse $1/y^2$ are indicated in
this diagram:
$$
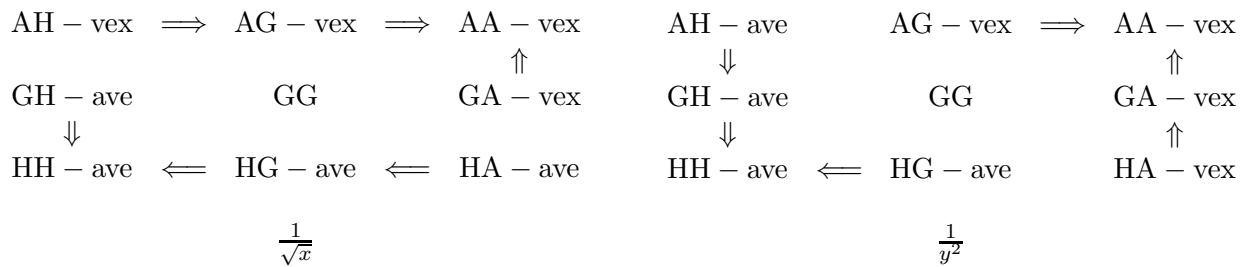
\begin{array}{ccccc}
{\rm AH-vex} & \Longrightarrow & {\rm AG-vex} & \Longrightarrow & {\rm AA-vex}\\
 & &  & & \Uparrow\\
{\rm GH-ave} &  & {\rm GG} &  & {\rm GA-vex}\\
\Downarrow& &  & & \\
{\rm HH-ave} & \Longleftarrow & {\rm HG-ave} & \Longleftarrow & {\rm HA-ave}\\
& & & & \\
& &\frac{1}{\sqrt{x}}& &\\ 
& & & & 
\end{array}
\qquad
\begin{array}{ccccc}
{\rm AH-ave} &    & {\rm AG-vex} & \Longrightarrow & {\rm AA-vex}\\
 \Downarrow& &  & & \Uparrow\\
{\rm GH-ave} &   & {\rm GG} &  & {\rm GA-vex}\\
\Downarrow& & & & \Uparrow \\
{\rm HH-ave} & \Longleftarrow & {\rm HG-ave} &  & {\rm HA-vex}\\
& & & & \\
& &\frac{1}{y^2}& &\\ 
& & & & 
\end{array}
$$

\medskip
For our application to the functions $\phi_1$ and $\phi_2$ we have the differentiablility needed to
apply Theorem~\ref{thm:thmpqqp}.
However, for other applications, we note that there are other proof techniques.
With notation as in that theorem, $g$ the inverse of $f$,  obviously for any invertible $\alpha$ and $\beta$,
$(\beta\circ{f}\circ\alpha)^{-1} = \alpha^{-1}\circ{g}\circ\beta^{-1}$.
Hence, for example,  with id the identity and recip the reciprocal function taking $x$ to $1/x$,
\begin{align*}
{\mbox{\rm convex, decreasing\ }} f\ {\mbox{\rm is HA-convex\ }}
&\Longleftrightarrow&
{\rm id}\circ{f}\circ{\rm recip} \ \ {\mbox{\rm is convex, increasing}} \\
&\Longleftrightarrow&
{\rm recip}\circ{g}\circ{\rm id} \ \ {\mbox{\rm is concave, increasing}}\\
&\Longleftrightarrow&
{\mbox{\rm convex, decreasing }} g\ {\mbox{\rm is AH-convex. \ }}
\end{align*}
\medskip

A few results anticipating parts of the preceding theorem have been published.
For example, concerning the inverse of a GG functions,
we have the following, adapted from~\cite{NiP04} Lemma 2.3.4, items denoted with an a are 
additions not explicitly in their Lemma 2.3.4.)

\begin{fact}\label{fac:NiP04} {\it
  If a function $f$ is increasing, multiplicatively convex (GG), and one-to-one, then its inverse is multiplicatively concave 
(and vice versa).\\
(GG-a) If a function $f$ is decreasing, multiplicatively convex (GG), and one-to-one, then its inverse is multiplicatively convex.\\
(HH-a) If a function $f$ is increasing, HH-convex, and one-to-one, then its inverse is HH-concave 
(and vice versa).\\
If a function $f$ is decreasing, HH-convex, and one-to-one, then its inverse is HH- convex.
}
\end{fact}

\begin{proof}
Once again, simple proofs for $C^2$ functions use the identities relating 
derivatives of $g$, the inverse of $f$,  to those of $f$.
\end{proof}

\section{Completely monotone and Stieltjes functions}\label{sec:CMStieltjes}

\subsection{Higher derivatives.  Introduction}\label{subsec:subsecHigher}

This subsection is an aside to the main function of this paper, namely to
establish generalised convexity/concavity properties of $\phi_1$ and $\phi_2$
and their inverses, for use in~\citeKWRob.
The subsection is included as the questions here,
besides being of interest in their own right, may, if answered,
provide simpler and neater routes to the properties used n~\citeKWRob.
\medskip

If one knows a function $f$ is completely monotone or even Stieltjes
this provides information about derivatives of all orders.
In the same way as we have treated first derivatives in~\S\ref{sec:secFirst}
and second derivatives in~\S\ref{sec:secSecond} one may obtain an expression for
the higher derivatives of the inverse, $g$, in terms of those of the original function $f$.
The formulae are given in~\cite{Jo02}.

There are many ways that completely monotone and Stieltjes functions
either determine relations involving their derivatives or are characterised by these.
Concerning third derivatives of Stieltjes functions, consider $n=3$ in equations~(\ref{eq:eqxpt2}) and~(\ref{eq:eqxpt3}).
For a Stieltjes function $f$, not only is it completely monotone, with the sign information on the derivatives of $f$,
but as $f_*(x)=f(1/x)/x$ is also Stieltjes , we also have sign information from the derivatives of $f_*$.
Concerning completely monotone functions we remark that a result concerning their Hankel determinants
is stated near the middle of p167 of~\cite{Wi41}.

We do not investigate higher derivatives here, but believe they may be useful in
answering questions like the following:

\begin{question}\label{qes:q2}
(i) {\it What conditions (if any) are needed to ensure that a Stieltjes function $f$ with range $(0,\infty)$ is
GA-convex?}\\
(ii) {\it What conditions (if any) are needed to ensure that the inverse $g$ of a Stieltjes function $f$ with range $(0,\infty)$ 
which is GA-convex is such that $g$ is
completely monotone?}\\
(iii) {\it Does $\phi_2$ satisfy these conditions?}
\end{question}

An affirmative answer to part (iii) would answer the first part of Question~\ref{que:q1}.
There are many inequalities satisfied by derivatives of Stieltjes functions,
some of which are given in inequalities~(\ref{eq:eqxpt1}) to~(\ref{eq:eqxpt3}).
Others are given in~\cite{Wi41}, etc.

\medskip
Perhaps an easier question is:

\begin{question}\label{qes:q3}
(i) {\it Let $f$ be a  completely monotone function mapping $(0,\infty)$ onto itself. 
Is  the inverse of $f$ also completely monotone?}\\
(ii) {\it If, as we expect, not, give an example.}
\end{question}

It is relatively easy to construct an example of a positive, decreasing, log-convex functions whose inverse is not log-convex.
If considerations of second derivatives do not suffice to answer Question~\ref{qes:q3}, or if higher derivatives make it easier to answer,
the reference~\cite{Jo02} has the relevant formulae.

\subsection{Completely monotone functions}

A function $f:(0,\infty)\rightarrow(0,\infty)$  is called {\it completely monotone} if f has derivatives of all orders and satisfies
$(-1)^n f^{(n)}(x)\ge{0}$ for all $x> 0$ and all npnnegative integers $n$. 
In particular, completely monotonic functions are decreasing and convex. 

\medskip


\noindent{\bf Lemma 3.4 of~\cite{Me12}\ }{\it
 (i) If $g'\in\CM$ then the function $x\mapsto\exp(-g(x)$ is $\CM$.\\
(ii) If $\log f$ is $\CM$, then $f$ is $\CM$ (the converse is not true).\\ 
(iii) If $f\in\CM$ and $g$ is a positive function with a $\CM$ derivative
(i.e. a Bernstein function), then the composition $x\mapsto~f(g(x))$ is $\CM$.}

\medskip
A particular case of Lemma 3.4(iii) applies to our functions $\phi_1$ and $\phi_2$.
As $\sqrt{x}$ is Bernstein, as $\phi_1\in\CM$ so also $\phi_1(\sqrt{x})=\phi_2(x)\in\CM$.
More generally,
since, for $t>0$, $\exp(-t\sqrt(z))$ is $\CM$, it follows from the Laplace transform representation
of $f_1(z)\in\CM$ that $f_1(\sqrt{z})$ is also in $\CM$.

Starting with a function $f_2(z)$ and obtaining properties of $f_2(z^2)$, 
i.e. $f_2\circ{\rm square}$, seems more difficult.
Of course, $\phi_1(z)=\phi_2(z^2)$ is $\CM$.
Starting with $\phi_1\in\CM$ then forming $\phi_1(z^2)$, we note that the inverse Laplace transform of the latter can be found,
and it changes sign.
$\phi_1(z^2)$ is not in $\CM$, but seems to be log-convex.
Starting from the Laplace representation of $f_2$ seems to be unhelpful in general. 
We remark that
$\exp(-t z^2)$ is log-concave in $z$ and sums of log-concave functions are not necessarily log-concave.
However, consider starting from the Stieltjes representation.
The function $1/(t+z^2)$ is $\CM$ in $z$.
Hence if $f$ is a Stieltjes function, $f(z^2)$ is $\CM$ in $z$,
as is the case with our functions $\phi_2$ and $\phi_1(z)=\phi_2(z^2)$.

\subsection{Stieltjes functions}\label{subsec:subsecStieltjes}

\subsubsection{Definition and basic properties}

Stieltjes functions ${\cal S}$ are a subclass of $\CM$.
A non-negative function $f$  is called a Stieltjes function ($f\in{\cal S}$ for short) if there exists 
$a_0\ge{0}$, $a_1\ge{0}$  and a non-negative measure $\mu(dt)$ on$[0,\infty)$ integrating 
$(z + t)^{-1}$ such that
\begin{equation}
 f(z) = a_0 + \frac{a_1}{z} + \int_0^\infty \frac{1}{z+t} \, \mu(dt) .	
 \label{eq:StDef}
 \end{equation}
While ${\cal S}$ is not closed under multiplication it is `logarithmically convex' in the sense that
for all $g_0,\ g_1\in{\cal S}$ and $\alpha\in(0,1)$ we have $g_0^{1-\alpha} g_1^\alpha \in{\cal S}$.
(See~\cite{SSV} Proposition~7.10.)
Various other cones of functions are treated in~\cite{SSV}.
The (nonzero) complete Bernstein functions ($\CBF$) are the reciprocals of
(nonzero) Stieltjes functions.

Some of the convexity properties of the kernel $1/(z+t)$ are given here:
$$
\begin{array}{ccccc}
{\rm AH-both} & \Longrightarrow & {\rm AG-convex} & \Longrightarrow & {\rm AA-convex}\\
\Downarrow & &  & &\\
{\rm GH-concave} & \Longleftarrow & {\rm GG-concave} &  & {\rm GA-neither}\\
\Downarrow& & \Downarrow & & \\
{\rm HH-concave} & \Longleftarrow & {\rm HG-concave} & \Longleftarrow & {\rm HA-concave}\\
& & & & \\
& &\frac{1}{1+x}& &\\ 
& & & & 
\end{array}
$$
The function $f(x)=1/(1+x)$ has as inverse the function $g$ defined on $(0,1]$ by
$g(y)=(1/y)-1$.
The function $g$ is log-convex when $y<1/2$ and log-/concave for $1/2<y\le{1}$.
The function $f(x)$ is GA-convex for $x>1$ and GA-concave for $0<x<1$.
One observation that follows from this is if the input $\mu(t)$ to the representation~(\ref{eq:StDef})
is larger for small $t$ than at larger  $t$ the Stieltjes function so formed is more
likely to be GA-convex.

\medskip

There are other characterizations of Stieltjes functions.\\
$\bullet$
Besides that of equation~(\ref{eq:StDef}) there is that of iterated Laplace transforms
(for which, see~\cite{SSV} Theorem 2.2 p12).
For  function $f$ to be a Stieltjes function its inverse Laplace transform should be
completely monotone.
A simple consequence of this is that if $f(z)$ is Stieltjes, then 
the completely monotone function $f(z)/z$ cannot be Stieltjes.
More generally if $f(z)$ is completely monotone, 
the completely monotone function $f(z)/z$ cannot be Stieltjes as
the inverse Laplace transform of the latter is the integral fro 0 of the inverse Laplace transform of the former,
so increasing and hence not completely monotone.\\
$\bullet$ 
There are also characterizations in terms of Nevalinna-Pick functions
(for which, see~\cite{SSV} p56).

\subsubsection{Examples, focusing on functions related to that in the application in \S\ref{sec:transcendental}}

Define, for $0{\le}a<b\le{\infty}$,
$$ St(a,b,u,z) = \int_a^b\frac{u}{z+t}\, dt . $$
Various examples of Stieltjes functions follow:
{\small
\begin{center}
\begin{tabular}{||c|c|c|c|c||}
\hline
 & St& $f(z)$& in terms of $\phi_2$&invlaplace($f$)\\
\hline
1& $St(0,1,\frac{1}{2\sqrt{t}},z)$& $ \frac{1}{\sqrt{z}} \, \arctan\left( \frac{1}{\sqrt{z}} \right)$& 
 $\phi_2(z)$ &$\frac{\sqrt{\pi}\, {\rm erf}(\sqrt{s})}{2\sqrt{s}}$  \\
  & & &  & ${_1}F_1(\frac{1}{2};\frac{3}{2};-s)$ \\
 \hline
2& $St(0,b,\frac{1}{2\sqrt{t}},z)$& $ \frac{1}{\sqrt{z}} \, \arctan\left( \frac{\sqrt{b}}{\sqrt{z}} \right) $&
   $\phi_2\left(\frac{z}{b}\right)/\sqrt{b}$ & \\
& & &  &  \\ 
\hline
3& $St(0,\infty,\frac{1}{2\sqrt{t}}, z)$ &$\frac{\pi}{2\sqrt{z} }$&
$\phi_2(z)+\frac{1}{z}\phi_2(\frac{1}{z})$  &$\frac{1}{2}\sqrt{\frac{\pi}{s}}$ \\
& & & &  \\
\hline
4& & $\frac{\arctan(\sqrt{z})}{z}$&
$\frac{1}{z\sqrt{z}}\, \phi_2(\frac{1}{z})$  & \\
& & &  &$\frac{\pi}{2} -2\sqrt{\frac{s}{\pi}}\, {_2}F_2(\frac{1}{2},1;\frac{3}{2}, \frac{3}{2}; -s)$ \\
\hline
5& $St(1,\infty,\frac{1}{2\sqrt{t}}, z)$ & $ \frac{\pi-2\arctan(1/\sqrt{z})}{2\sqrt{z} }$ &
$\frac{\pi}{2\sqrt{z}} -\phi_2(z)$  &$\frac{\sqrt{\pi}\, {\rm erfc}(\sqrt{s})}{2\sqrt{s}}$ \\
& &  $\frac{\arctan(\sqrt{z})}{\sqrt{z}}$&
 $\frac{1}{z}\,\phi_2(\frac{1}{z}) $  &$\frac{1}{2}\sqrt{\frac{\pi}{s}}-{_1}F_1(\frac{1}{2};\frac{3}{2};-s)$ \\
& & &  &  \\
\hline\hline
6& $St(0,1,\frac{\sqrt{t}}{2},z)$ &$1-\sqrt{z}\arctan(\frac{1}{\sqrt{z}})$& 
$1 - z \phi_2(z)$ & \\
& & &  &  \\
\hline
7& $St(0,\infty,\frac{\sqrt{t}}{1+t},z)$&$\frac{\pi}{1+\sqrt{z}}$&
 & $\sqrt{\frac{\pi}{s}} -\pi\exp(s){\rm erfc}(\sqrt{s})$  \\
& & &  &$\sqrt{\frac{\pi}{s}} -\pi\exp(s)\left(1-2\sqrt{\frac{s}{\pi}},{_1}F_1(\frac{1}{2};\frac{3}{2}; -s )\right)$  \\
\hline
8& $St(0,1,\frac{\sqrt{t}}{1+t},z)$&$\frac{2\sqrt{z}\arctan(1/\sqrt{z})-\frac{\pi}{2}}{-1+z}$&
$\frac{2 z\phi_2(z)-\frac{\pi}{2}}{-1+z}$  &  \\
& & &  &  \\
\hline
9& & $\arctan(\frac{1}{\sqrt{z}})$&
$\sqrt{z}\, \phi_2(z)$  &$\frac{\exp(-s)\,{\rm erfi}(\sqrt{s})}{2s}$ \\
& &  & & $\frac{\exp(-s)}{\sqrt{s\pi}} \,{_1}F_1(\frac{1}{2};\frac{3}{2}; s )$ \\
\hline
10&$St(0,\infty,\delta(t-1),z)$&$\frac{1}{1+z}$&
 &  $\exp(-s)$ \\
& &  & &  \\
\hline
\end{tabular}
\end{center}
} 

\medskip

Some comments on the table are appropriate.
The first few entries all have range $(0,\infty)$.
The later entries have finite ranges (and some, e.g. item 10, are not GA-convex).
\begin{itemize}
\item
Entry 5 for $St(1,\infty,\frac{1}{2\sqrt{t}}, z)$ checks against
$$ \arctan(\frac{1}{\sqrt{z}}) + \arctan({\sqrt{z}})= \frac{\pi}{2}. $$
Also, as mentioned before, for any $f\in{\cal S}$, we have that
$x\mapsto{f(1/x)/x}$ is also in $\cal S$, checking against:
$$  \frac{\arctan(\sqrt{z})}{\sqrt{z}} =\frac{1}{z}\,\phi_2(\frac{1}{z}) . $$

\item
Entry 9's function $\arctan(1/\sqrt{z})$ can be seen to be a Stieltjes function as it is the
Laplace transform of a completely monotonic function:
$$\arctan(\frac{1}{\sqrt{z}}) 
= \int_0^\infty \exp(-z t)\, \frac{\exp(-t)\, {\rm erfi}(\sqrt{t})}{2 t} \, dt .$$
The complete monotonicity is proved via the following steps.
Beginning from the definition of erfi and 
using $v=(t-s^2)/t$ as a change of variable,  we find
\begin{align*}
\frac{\exp(-t)\, {\rm erfi}(\sqrt{t})}{2 t} 
&=\frac{1}{\sqrt{\pi}}\int_0^{\sqrt{t}} \frac{\exp(s^2-t)}{t}\, ds \\
&= \frac{1}{2\sqrt{\pi}}\int_0^1 \frac{\exp(-vt)}{\sqrt{t}\,\sqrt{1-v}} \, dv
\end{align*}
However, for $v>0$, $\exp(-v t)/\sqrt{t}$ is the product of $\CM$ functions, so $\CM$,
and sums and integrals of $\CM$ functions are $\CM$.
Hence $\arctan(1/\sqrt{z})$ is Stieltjes.

\item
Entry 4 follows from entry 9 on using Property (ii) from~\S\ref{subsubsec:VPropStieltjes}.

\item
In entry 10, $\delta$ is the Dirac delta measure.\\
Entry 10 is well known to be Stieltjes and from it one notes that $-z\phi_2'(z)$ is Stieltjes, as
$$ -z\phi_2'(z) = \frac{1}{2}\left( \phi_2 + \frac{1}{1+z} \right) . $$

\end{itemize}

\subsubsection{Various properties of Stieltjes functions}\label{subsubsec:VPropStieltjes}

Here is a short list of some properties of the cone of Stieltjes functions:
\begin{enumerate}
\item[(i)] $ f {\in} {\cal S}\setminus \{0\} \Longrightarrow \frac{1}{f(1/x)} {\in} {\cal S}$,
i.e. ${\rm recip}\circ{f}\circ{\rm recip}{\in} {\cal S}$
\item[(ii)] $f{\in}{\cal S}\setminus \{0\}  \Longrightarrow \frac{1}{xf (x)} {\in} {\cal S} $
\item[ ] From these,  $ f {\in} {\cal S}\setminus \{0\} \Longrightarrow {f(1/x)}/x {\in} {\cal S}$
\item[(iii)] $f{\in}{\cal S}, \lambda>0 \Longrightarrow \frac{f}{ \lambda f +1} {\in}{\cal S} $
\item[(iv)] $f,g{\in}{\cal S}\setminus \{0\}  \Longrightarrow f\circ\frac{1}{g} ,\ \frac{1}{f\circ{g}} {\in}{\cal S} $
\item[(v)]  $f,g{\in}{\cal S}, \ 0<\alpha<1  \Longrightarrow f^\alpha\circ g^{1-\alpha} {\in}{\cal S}$
\item[(vi)] $ f \in {\cal S}, \ 0<\alpha<1  \Longrightarrow f^\alpha  \in{\cal S}$.
\end{enumerate}

\subsubsection{Stieltjes functions, AG and GA}

Any completely monotone function is log-convex, i.e. AG-convex.

We have yet to check when (if always) a Stieltjes function with range $(0,\infty)$ is GA-convex.
The Stiieltjes function $\phi_2$ is GA-convex.
Any Stieltjes function which is the inverse of a log-convex function,
e.g. a completely monotone function, is GA-convex.

\subsubsection{Stieltjes functions, AH and HA}

\begin{theorem}\label{thm:StieltjesRecip}
For any Stieltjes function $\phi$, $(1/\phi)\in\CBF$ so, in particular, $1/\phi$ is concave, or,
In other words, $\phi$ is AH-convex.

\end{theorem}
\begin{proof}
This follows as Theorem~7.3 of~\cite{SSV} ensures that $1/\phi$ is a complete Bernstein function.
Also any Bernstein function $b$ is positive with $b'\in\CM$, so we have $(1/\phi)'\in\CM$ is positive and decreasing.
That it is decreasing is $(1/\phi)''<0$, i.e. $1/\phi$ is concave.
In other words, $\phi$ is AH-convex.
\end{proof}

\medskip

\begin{theorem}\label{thm:Stieltjes1ox}
For any Stieltjes function $\phi$, the function $x\mapsto\phi(1/x)$ is concave, or,
in other words, $\phi$ is HA-concave.
\end{theorem}

\begin{proof}
The result of Theorem~\ref{thm:Stieltjes1ox} follows from item (i) in the above list and Theorem~\ref{thm:StieltjesRecip}.
See also~\cite{Me04}.
\end{proof}
\medskip

In connection with our later application, we remark that Theorem~\ref{thm:Stieltjes1ox} gives that
$\phi_2(1/x)$ is concave, whereas $\phi_1$, which is not Stieltjes, is such that $\phi_1(1/x)$ is convex.

\medskip
Stieltjes functions are simultaneously AH-convex and HA-concave.
This, with the AH-HA case of Theorem~\ref{thm:thmpqqp}, gives another proof of Corollary~\ref{cor:corfac3}:
the only Stieltjes functions whose inverses are Stieltjes
are positive multiples of $1/x$.

\subsubsection{Stieljes functions, GG, HG and HH}

The next result is weaker than the HA-concavity of Theorem~\ref{thm:Stieltjes1ox}
which implies HG-concavity which, in turn, implies HH-concavity:

\begin{theorem}\label{thm:StieltjesHH}
(i) Any Stieltjes function is HG-concave.\\
(ii) Any Stieltjes function is HH-concave.
\end{theorem}

\begin{proof}
(ii) Perhaps item (ii) is the easier. There is a one-line proof:
$$f\in{\cal S} \Longleftrightarrow
\frac{1}{f(\frac{1}{x})}\in{\cal S} \Longrightarrow
\frac{1}{f(\frac{1}{x})}\ \ {\rm\ is\ convex}\ \  
 \Longleftrightarrow f \ \ \ {\rm\ is\ HH-concave} .$$
 See~\cite{AVV07} Theorem 2.4(9), with calculus proofs using 2.5(9).\\
 (i) The one-line can be adapted:
 $$f\in{\cal S} \Longleftrightarrow
\frac{1}{f(\frac{1}{x})}\in{\cal S} \Longrightarrow
-\log({f(\frac{1}{x})})\ \ {\rm\ is\ convex}\ \  
 \Longleftrightarrow f \ \ \ {\rm\ is\ HG-concave} .$$
 \end{proof}
 
 \medskip
 
  We remark also that
 while the Stieltjes function $\phi_2$ is necessarily HH-concave,
 the function $\phi_1$, which is merely in $\CM$, is HH-convex.
 
 \medskip
The Stiieltjes function $\phi_2$ is GG-concave.
The Stieltjes function $1/\log(1+x)$ (see~\cite{SSV} p228 entry 26)
is GG-convex.

\section{A transcendental equation}\label{sec:transcendental}

The transcendental equation
$$ X\tan(X)=Y , \qquad {\rm with\ \ } Y>0 $$
and an interest in solutions $X$ with $0<X<\pi/2$ arises in various applications.
The purpose of the remainder of this paper is to extract information on its solutions in a form that
can be used in our subsequent paper~\citeKWRob.
Before doing this, we note that there are other applications.
This transcendental equation has been widely studied, e.g.~\cite{BS73},~\cite{MRS03},~\cite{LWH15}.
 Numerical values, often used for checks, are given in Table 4.20 of~\cite{AS65}.
Amongst the applications, other than ours in~\citeKWRob,  are (i) the energy spectrum for the 
one-dimensional quantum mechanical finite square well,
and (though with $c<0$)
(ii) (though with $c<0$) zeros of the spherical Bessel function $y_1(x)=j_{-2}(x)$.

In the application, and notation,  in~\citeKWRob the problem is given $\beta>0$,
how does $\mu$ depend on $c$ where $\mu(c)$ solves the transcendental equation:
\begin{equation}
 \mu\tan(c\mu)= \frac{1}{\beta} .
 \label{eq:muc}
 \end{equation}
 It happens that one can re-scale variables so that there is just one independent variable ${\hat c}$:
  \begin{equation}
{\hat\mu} \tan({\hat c}{\hat\mu})=1,
 \qquad{\mbox{\rm where\ \ }} {\hat\mu}=\beta\mu,\ {\hat c}=\frac{c}{\beta} \  .
 \label{eq:rectmuchat}
 \end{equation}

 We have an interest in the smallest positive solutions,
\begin{displaymath}0<\mu({c})<\pi/(2 c),\qquad   0<{\hat \mu}<\pi/(2 {\hat c}) .
\end{displaymath}

In the application in~\citeKWRob much of the effort involves obtaining results valid for $\beta\ge{0}$ --
for Robin boundary conditions --
where the corresponding result with $\beta=0$  --
for Dirichlet boundary conditions -- has been available for decades.

We will, henceforth,  also drop the hat notation. 
In the next subsection we explore the behaviour of the function $\phi_1(\mu)$,\ defined as in~(\ref{eq:phiarctan}),
 which is inverse to $\mu(c)$, that is
 $$ \mu \tan(\mu \phi_1(\mu)) = 1 . $$
 Also explored are the convexity properties of $\phi_2$ where $\phi_2(z)=\phi_1(\sqrt{z})$.
 
 \subsection{The convexity properties of $\phi_2$ and $\phi_1$}\label{subsec:Propphi}
 
 Some properties follow from complete monotonicity, and, for $\phi_2$ others
 follow from it being a Stieltjes function.
 Yet further properties follow from calculation
 (the details of which are relegated to an appendix).
 
 \begin{theorem}
\label{lem:lem1}{\it Both $\phi_1$ and $\phi_2$ are completely monotone.
 Furthermore $\phi_2$ is a Stieltjes function.
}
\end{theorem}

\begin{proof}
We have
\begin{equation} \phi_1(z) = \int_0^\infty \exp(-z t) {\rm Si}(t)\, dt ,
\label{nu:phi1ILT}
\end{equation}
where {\rm Si} is the sine integral
\begin{equation} {\rm Si}(t) = \int_0^t \frac{\sin(\tau)}{\tau} \, d\tau .
\label{nu:SiDef}
\end{equation}
Since ${\rm Si}(t)>0$ for $t>{0}$, $\phi_1(z)$ is completely monotone.

The same Laplace transform representation also shows $\phi_2$ to be completely  monotone as
\begin{equation} \phi_2(z)
= \int_0^\infty \exp(-z t)\, \frac{1}{2} \sqrt{\frac{\pi}{t}} {\rm erf}(\sqrt{t})\, dt ,
\label{nu:phi2LT}
\end{equation}
and the integrand in the expression above is positive. Furthermore 
$$  \sqrt{\frac{\pi}{t}} {\rm erf}(\sqrt{t})
= \int_0^1 \frac{\exp(-s t)}{\sqrt{s}}\, ds ,
$$
which is completely monotone, so by~\cite{SSV}~Theorem 2.2(i), $\phi_2\in{\cal S}$.
That $\phi_2$ is a Stieltjes function also follows, 
as we have already noted in the table of examples in~\S\ref{subsec:subsecStieltjes}
$$ \phi_2(z)
=\frac{1}{\sqrt{z}} \, \arctan\left( \frac{1}{\sqrt{z}} \right)
=\int_0^1 \frac{1}{z+t}\, \frac{dt}{2\sqrt{t}} ,
$$
This completes the proof.
\end{proof}

\medskip

A calculation gives
$$
\frac{d}{dx} \left(x^{1-p} \phi_2'(x)(\phi_2(x))^{q-1}\right)
= \frac{x^{-1-p}\phi_2^{q-2}}{4 (1+x)^2} Q_2(\phi_2,x) $$
where 
\begin{align*}
Q_2(\Phi_2,x)&=& (1+x)^2 (2p+q) \Phi_2^2 +(2(p+q)(1+x)+x-1)\Phi_2 +(q-1) \\
 &=& (2p+q) \Phi_2^2 x^2 +(2 (2p+q)(\Phi_2+1)+1-2p)\Phi_2 x +(\Phi_2+1)((2p+q)\Phi_2 +q-1) .
\end{align*}
We establish the convexity/concavity properties by establishing that $Q_2(\phi_2(x),x)$ does not
change sign.
The details of the calculation are relegated to an Appendix.
The results are summarised in the following diagram.

$$
\begin{array}{ccccc}
{\rm AH-convex} & \Longrightarrow & {\rm AG-convex} & \Longrightarrow & {\rm AA-convex}\\
 & &  & & \Uparrow\\
{\rm GH-concave} & \Longleftarrow & {\rm GG-concave} &  & {\rm GA-convex}\\
\Downarrow& & \Downarrow & & \\
{\rm HH-concave} & \Longleftarrow & {\rm HG-concave} & \Longleftarrow & {\rm HA-concave}\\
& & & & \\
& &\phi_2& &\\ 
& & & & 
\end{array}
$$
From the diagram above one expects (correctly) that the hardest results to establish would be
GA-convexity,  GG-concavity and HA-concavity.

The corresponding results for $\phi_1$, also established in the appendix, are:
$$
\begin{array}{ccccc}
{\rm AH-concave} &    & {\rm AG-convex} & \Longrightarrow & {\rm AA-convex}\\
 \Downarrow& &  & & \Uparrow\\
{\rm GH-concave} & \Longleftarrow & {\rm GG-concave} &  & {\rm GA-convex}\\
\Downarrow& & \Downarrow & & \Uparrow \\
{\rm HH-concave} & \Longleftarrow & {\rm HG-concave} &  & {\rm HA-convex}\\
& & & & \\
& &\phi_1& &\\ 
& & & & 
\end{array}
$$

\subsection{Properties of $\mu$ used in~\citeKWRob}\label{subsec:Propmu}

Reversing the letters in the displays for $\phi_1$ and for $\phi_2$ gives the following.
\begin{itemize}
\item
The convexity properties of $\mu$ are represented thus:
$$
\begin{array}{ccccc}
{\rm HA-concave} &    & {\rm GA-convex} & \Longrightarrow & {\rm AA-convex}\\
 \Downarrow& &  & & \Uparrow\\
{\rm HG-concave} & \Longleftarrow & {\rm GG-concave} &  & {\rm AG-convex}\\
\Downarrow& & \Downarrow & & \Uparrow \\
{\rm HH-concave} & \Longleftarrow & {\rm GH-concave} &  & {\rm AH-convex}\\
& & & & \\
& &\mu& &\\ 
& & & & 
\end{array}
$$
We remark that the properties are the same as in the corresponding diagram for the Stieltjes function $\phi_2$.
Theorems~\ref{thm:StieltjesRecip} and~\ref{thm:Stieltjes1ox} ensure that any Stieljes function is
both AH-convex and HA-concave.
We have no information yet to preclude the possibility that $\mu$ is Stieltjes, but as we have no proof
that it is even completely monotone, it is too early to speculate.
\item
The convexity properties of $\mu_{(2)}$ are represented thus:
$$
\begin{array}{ccccc}
{\rm HA-convex} & \Longrightarrow & {\rm GA-convex} & \Longrightarrow & {\rm AA-convex}\\
 & &  & & \Uparrow\\
{\rm HG-concave} & \Longleftarrow & {\rm GG-concave} &  & {\rm AG-convex}\\
\Downarrow& & \Downarrow & & \\
{\rm HH-concave} & \Longleftarrow & {\rm GH-concave} & \Longleftarrow & {\rm AH-concave}\\
& & & & \\
& &\mu_{(2)}& &\\ 
& & & & 
\end{array}
$$
Clearly $\mu_{(2)}$ is not Stieltjes. If $\mu$ were to be shown to be completely monotone,
then so is its square, $\mu_{(2)}$.
\end{itemize}

As $\mu_{(2)}=\mu^2$ there are some obvious checks.
For example, it is clear that the AG, GG- and HG-convexity properties 
of  $\mu_{(2)}$ and of $\mu$ must be the same.
The convexity properties that differ are AH and HA.

\section{Other transcendental equations involving $\CM$ functions}\label{sec:Othertr}

Denote the Lambert W function by $W$.
Results concerning Stieltjes representations of $W$ are given in~\cite{KJCB12}.\\
The Stieltjes function $f(x)=W(1/x)$ is the solution of $g(f):=\exp(-f)/f=x$, and $g\in\CM$ as
$$ g(y)=\frac{\exp(-y)}{y}
=\int_1^\infty \exp(-t y)\, dt .$$
The Stieltjes function $f(x)=1/W(x)$ is the solution of $g(f):=\exp(1/f)/f=x$, and $g\in\CM$ as
$$ g(y)=\frac{\exp(1/y)}{y}
=\int_0^\infty {\rm BesselI}(0,2\sqrt{t})\exp(-t y)\, dt .$$

The function
$$ g(y) =\frac{\tanh(\sqrt{y})}{\sqrt{y}}\  \in {\cal S} ,$$
has arisen in connection with one of the author's applied mathematical interests -- water waves.
That $g\in{\cal S}$ is from Proposition 2.22 of~\cite{EW16}.
Amongst the ever-growing menangerie of special functions are 
`generalized Lambert W'  functions.
The solution of $g(y)=x$ is given in~\cite{MK16} as
\begin{equation}
\sqrt{y} = \frac{1}{2}
W\left(\begin{array}{c}2/x\\-2/x\end{array}; -1\right) .
 \label{eq:KCSols}
\end{equation}
To the best of the author's knowledge, there has been no systematic study of
complete monotonicity properties of generalized Lambert W functions.

Here is another example involving a generalized Lambert function,
this time with two upper parameters.
The equation to be solved, for $x$ is
\begin{equation}
\exp(x y)= (1+x/a_1) (1+x/a_2) \ \ {\rm with\  \ } a_1>0, \ a_2>0 .
\label{eq:W2up}
\end{equation}
When $a>0$, $\log(1+z/a)/z$ is a Stieltjes function for $z>0$.
Hence it is completely monotone, so convex, and log-convex.
A version of the generalized Lambert function - with 2 upper parameters - arises
in solving equation~(\ref{eq:W2up}).
One is interested in $x(y$). Take logs
$$y = \log((1+x/a_1)\,(1+x/a_2))/x .$$
Now the expression on the right is a Stieltjes function of $x$, and, from this
one can draw some conclusions concerning the convexity properties of $x(y)$.
(This may be related to a physical problem.
See~\cite{Bar16}, equations (4) and (62)-(63).)


\section{Conclusion}\label{sec:questions}

As a consequence of research associated with convexity properties of a domain functional
from a partial differential equation  problems (see~\citeKWRob)
various theorems associated with convex functions
and their inverses were discovered.
Theorem~\ref{thm:thmpqqp} is the simplest of these.

The application involved a function $\phi_2$ which was instantly noted to be Stieltjes from which
AH-convexity and HA-concavity (and consequences) follow immediately.
Several open questions are posed.  See~\S\ref{subsec:subsecHigher}.
That which seems the most important for further results
associated with the application of~\citeKWRob is whether the inverse of $\phi_2$ 
is completely monotone.

\section*{A: Calculations concerning $\phi_2$ and $\phi_1$}

A calculation gives
$$
\frac{d}{dx} \left(x^{1-p} \phi_2'(x)(\phi_2(x))^{q-1}\right)
= \frac{x^{-1-p}\phi_2^{q-2}}{4 (1+x)^2} Q_2(\phi_2,x) ,$$
where 
\begin{align*}
Q_2(\Phi_2,x)&= (1+x)^2 (2p+q) \Phi_2^2 +(2(p+q)(1+x)+x-1)\Phi_2 +(q-1) , \\
 &= (2p+q) \Phi_2^2 x^2 +(2 (2p+q)(\Phi_2+1)+1-2p)\Phi_2 x +(\Phi_2+1)((2p+q)\Phi_2 +q-1) .
\end{align*}
We establish the convexity/concavity properties by establishing that $Q_2(\phi_2(x),x)$ does not
change sign.
We do this by considering $\Phi_2$ as an independent variable in $Q_2(\Phi_2,x)$ 
and investigating this as $\Phi_2$ varies between lower and upper bounds of $\phi_2(x)$,
for which we use the interval $[3/(1+3x),1/x]$.

The upper bound is an obvious consequence of $\arctan(x)<x$. 
Establishing the lower bound can begin with
$$\frac{d}{dx}\left(\arctan(x) - \frac{x}{\frac{x^2}{3}+1}\right)
= \frac{4 x^4}{(1+x^2)(3+x^2)^2} \ge{0} .$$
On integrating the left-hand side from 0 we have
$$ \arctan(x) >  \frac{x}{\frac{x^2}{3}+1} .$$
From this
$$ \phi_2(x) >  \frac{1}{\frac{1}{3}+x} . $$
(The weaker inequality $\phi_2(x)>1/(1+x)$ suffices for all except for testing AH-convexity.)

The values of $Q_2$ at the end-points of the interval $[3/(1+3 x),1/x]$ are
\begin{align*}
Q_2(\frac{3}{1+3x},x) 
&=   \frac{4 Q_{2-}(p,q)}{(1+3x)^2}, &Q_{2-}(p,q) &= 9(p+q)x^2 +3(5p+4q-1)x+6p+4q-1 ,\\
Q_2(\frac{1}{x},x) 
&=   \frac{Q_{2+}(p,q)}{x^2},    &Q_{2+}(p,q) & = 4(p+q)x^2 + (6p+4q-1)x+(2p+q) .
\end{align*}
The function $\phi_2$ is $(p,q)$ convex (concave) iff $Q_2(\phi_2(x),x)>0$ (resp. $Q_2(\phi_2(x),x)<0$).
We are able to determine the sign of $Q_2(\Phi_2,x)$
over the interval $[3/(1+3x),1/x]$
(and often over a much larger interval of $\Phi_2$, which, however, is irrelevant to our needs).
For the sign to remain constant it is, of course, necessary that both
$Q_{2-}$ and $Q_{2+}$ have the same sign, and we record these in the table below.

\begin{center}
\begin{tabular}{|c|c|c|c|}
\hline
$(p,q)$& MN& $Q_2(\Phi_2,x)$&Notes \\
 & & &$[Q_{2-}, Q_{2+}]$\\
\hline
(1,1)&    AA-convex&  $\Phi_2( 3(1+x)^2\Phi_2 +5x+3)$& all terms in $Q_2$ positive \\
        &          &  &$[18x^2+24x+9,8x^2+9x+3]$ \\
        \hline
(1,0)&    AG-convex& $ 2(1+x)^2\Phi_2^2+(3x+1)\Phi_2-1$ & $Q_2$ positive for $\Phi_2>1/(1+x)$ \\
        &          &   & $[ 9x^2+12x+5 ,4x^2+5x+2  ]$ \\
        \hline
(1,-1)&   AH-convex& $ (1+x)^2\Phi_2^2+(x-1)\Phi_2-2$ &  sign change below interval \\
         &          &   &$[1  , 1+x ]$  \\
         \hline
(0,1)&    GA-convex&  $\Phi_2( (1+x)^2\Phi_2 +3x+1)$& all terms in $Q_2$ positive   \\
         &          & &   $[ 9x^2+9x+3 ,4x^2+3x+1  ]$\\
         \hline\hline
(0,0)&   GG-concave&  $x\Phi_2-\Phi_2-1$&  Use $\phi_2<1/x$ in first term \\
         &          &   &   $[ -3x-1 ,-x  ]$ \\
         \hline
(0,-1) &    GH-concave & $-(1+x)^2 \Phi_2^2-(x+3)\Phi_2-3$& all terms in $Q_2$ negative \\
          &          & &   $[ -9x^2-15x-5 ,-(x+1)(4x+1)  ]$\\
          \hline
(-1,1)&   HA-concave&  $-(1+x)^2 \Phi_2^2-(x-1)\Phi_2$&   Use $\phi_2>1/(1+x)$ in the first term and\\
& & & $\phi_2<1/x$ in the $x\Phi_2$ term\\
         &          &   &   $[ -6x-3,-3 x-1  ]$ \\
         \hline
(-1,0)&   HG-concave& $-2(1+x)^2 \Phi_2^2-(x+3)\Phi_2-1$ &  all terms in $Q_2$ negative \\
          &        &  &   $[ -9x^2-18x-7 ,-4x^2-7x-2  ]$ \\
          \hline
(-1,-1)&  HH-concave& $-3(1+x)^2 \Phi_2^2-(3x+5)\Phi_2-2$ &  all terms in $Q_2$ negative   \\
         &          &  &   $[-18x^2-30x-11  ,-(x+1)(8x+3)  ]$ \\
\hline
\end{tabular}
\end{center}

\bigskip

We now apply the same process to $\phi_1$.
One result that is used in our application~\citeKWRob is that $\phi_1$ is AH-concave and HA-convex,
both of which are different to $\phi_2$.
A calculation gives
$$
\frac{d}{dx} \left(x^{1-p} \phi_1'(x)(\phi_1(x))^{q-1}\right)
= \frac{x^{-1-p}\phi_1^{q-2}}{ (1+x^2)^2} Q_1(\phi_1,x) $$
where 
$$
Q_1(\Phi_1,x)
= (1+x^2)^2 (p+q) \Phi_1^2 +((p+2q)(1+x^2)+x^2-1)\Phi_1 +(q-1) .
$$
We establish the convexity/concavity properties by establishing that $Q_1(\phi_1(x),x)$ does not
change sign.
We do this by considering $\Phi_1$ as an independent variable in $Q_1(\Phi_1,x)$ 
and investigating this as $\Phi_1$ varies between lower and upper bounds of $\phi_1(x)$,
for which we use the interval $[3/(1+3x^2),1/x^2]$.
The values of $Q_1$ at the end-points are:
\begin{align*}
Q_1(\frac{3}{1+3x^2},x) 
&= \frac{2Q_{1-}(p,q)}{(1+3x^2)^2}, &Q_{1-}(p,q) &= 9(p+2q)x^4+3(5p+6q-2)x+6p+8q-2 ,\\
Q_1(\frac{1}{x^2},x) 
&=   \frac{Q_{1+}(p,q)}{x^4}, & Q_{1+}(p,q) &= (2p+4q)x^4 + (3p+4q-1)x+(p+q) .
\end{align*}
The function $\phi_1$ is $(p,q)$ convex (concave) iff $Q_1(\phi_1(x),x)>0$ (resp. $Q_1(\phi_1(x),x)<0$).
We are able to determine the sign of $Q_1(\Phi_1,x)$
over the interval $[3/(1+3x^2),1/x^2]$
(and often over a much larger interval of $\Phi_1$, which, however, is irrelevant to our needs).
For the sign to remain constant it is, of course, necessary that both
$Q_{1-}$ and $Q_{1+}$ have the same sign, and we record these in the table below.

\begin{center}
\begin{tabular}{|c|c|c|c|}
\hline
$(p,q)$& MN& $Q_1(\Phi_1,x)$&Notes \\
 & & &$[Q_{1-}, Q_{1+}]$\\
\hline
(1,1)&    AA-convex&  $\Phi_1( 2(1+x^2)^2\Phi_1 +4x^2+2)$& all terms in $Q_1$ positive \\
        &          &  &$[27x^4+33x^2+12,6x^4+6x^2+2]$ \\
        \hline
(1,0)&    AG-convex& $ (1+x^2)^2\Phi_1^2+2x^2 \Phi_1-1$ & $Q_1$ positive for $\Phi_1>1/(1+x^2)$ \\
        &          &   & $[ 9x^4+9x^2+4 ,2x^4+2x^2+1  ]$ \\
        \hline
(1,-1)&   AH-concave& $ -2\Phi_1-2$ &  all terms in $Q_1$ negative\\
         &          &   &$[-9x^4-15x^2-4  , -2x^2(x^2+1) ]$ \\
         \hline
(0,1)&    GA-convex&  $\Phi_1( (1+x^2)^2\Phi_1 +3x^2+1)$& all terms in $Q_1$ positive   \\
         &          & &   $[ 18x^4+18x^2+6 ,4x^4+3x^2+1  ]$\\
         \hline
(0,0)&   GG-concave&  $x^2\Phi_1-\Phi_1-1$&  Use $\phi_1<1/x^2$ in first term \\
         &          &   &   $[ -6x^2-2 ,-x^2  ]$ \\
         \hline
(0,-1) &    GH-concave & $-(1+x^2)^2 \Phi_1^2-(x^2+3)\Phi_1-2$& all terms in $Q_1$ negative \\
          &          & &   $[ -18x^4-30x^2-10 ,-(x^2+1)(4x^2+1)  ]$\\
          \hline
(-1,1)&   HA-convex&  $2x^2\Phi_1$&   all terms in $Q_1$ positive\\
         &          &   &   $[ 3x^2(3x^2+1),2x^4  ]$ \\
         \hline
(-1,0)&   HG-concave& $-(1+x^2)^2 \Phi_1^2-2\Phi_1-1$ &  all terms in $Q_1$ negative \\
          &        &  &   $[ -9x^4-21x^2-8 ,-2x^4-4x^2-1  ]$ \\
          \hline
(-1,-1)&  HH-concave& $-2(1+x^2)^2 \Phi_1^2-(2x^2+4)\Phi_1-2$ &  all terms in $Q_1$ negative   \\
         &          &  &   $[-27x^4-45x^2-16  ,-(x^2+1)(6x^2+2)  ]$ \\
\hline
\end{tabular}
\end{center}

The results are summarised in the diagrams in~\S\ref{subsec:Propphi}.
\newpage

\begin{center}
{\Large Supplement to Part I}
\end{center}

\section*{Direct calculation for `convexity' properties of $\mu$}

\noindent
{This section
is, now, probably just history, concerning early methods of getting the results.
However, there remains the possibilty that some neat formulae for the
$n$-th derivative of $\mu(c)$ might be found, allowing some inductive proof that
the derivatives alternate in sign.)}

\medskip

As before, we use the notation $\mu_{(2)}=\mu^2$ and ${\hat\mu}_{(2)}={\hat\mu}^2$.
This tidies the notation when differentiating with respect to such variables.

\medskip

A short calculation establishes that
with
$$ r(f):= \frac{f f''}{(f')^2} $$
we have
$$\frac{d^2}{d x^2} f(x)^k
= \frac{k f^{k-2}}{(f')^2}\left( r(f)+k-1\right) . $$
Hence
$$ r(f) > 1-k\quad\Longleftrightarrow k f^k{\rm\ is\ convex},\qquad
r(f) < 1-k\quad\Longleftrightarrow k f^k{\rm\ is\ concave} . $$

The ratio $r$ relates to power-concavity, and we can compare
power concavity of $f^k$ with that of $f$ in a natural way.
For any function of x, denoted $n$ here,
\begin{equation}
r(f):= \frac{f f''}{(f')^2} \qquad {\rm gives\ \ }
\frac{r(f^k)-n}{r(f)-(k n-k+1)}= \frac{1}{k},\quad{\rm and\ }
r(f^k)=\frac{1}{k}\left(r(f)  + k-1\right).
\label{eq:rDefProp}
\end{equation}
In particular
$$r(\mu_{(2)})=\frac{1}{2}\left( r(\mu)+1\right)\qquad{\rm and}\qquad
r(\frac{1}{\mu})= 2 - r(\mu) .
$$

\medskip

The functions $\mu(c)$, $\mu_{(2)}(c)$, (and ${\hat\mu}({\hat{c}})$, etc.)  are decreasing, convex, and indeed log-convex 
in $c$ (and in ${\hat{c}}$, respectively).

This is checked by routine differentiation.
At fixed $\beta>0$, $\mu$ decreases as $c$ increases:
 \begin{equation}
0> \frac{d  \mu}{d  c}
 = - \ \frac{\mu ( 1+ \beta^2\mu^2)}{\beta +c (1+\beta^2\mu^2)} >-\frac{\mu}{c},\qquad
 \frac{d  \mu_{(2)}}{d  c}
 = - \ \frac{2\mu_{(2)} ( 1+ \beta^2\mu_{(2)})}{\beta +c (1+\beta^2\mu_{(2)})} .
\label{eq:dmu2dc}
\end{equation}
The hat de's are obviously related to these, as, for example,  the $\mu_{(2)}$ de can be written
$$ \frac{d  {\hat\mu}_{(2)}}{d {\hat c}}
= - \ \frac{2{\hat\mu}_{(2)} ( 1+ {\hat\mu}_{(2)})}{1 +{\hat c} (1+{\hat\mu}_{(2)})} .
$$
Also $\mu(c)$ is convex in $c$:
$$\frac{d ^2 \mu}{d  c^2}
 =  \ \frac{2\mu ( 1+ \beta^2\mu^2)(2\beta^3 \mu^2 +c \beta^4 \mu^4 +2 c \beta^2 \mu^2 + \beta + c)}
 {(\beta +c (1+\beta^2\mu^2))^3} .
 $$
This gives the ingredients for calculating $r(\mu)$:
$$ r(\mu) = 2 - \frac{2\beta^3 \mu \mu'}{(1+\beta^2\mu^2)^2} ,$$
so that $r(\mu)>2$ and hence $1/\mu$ is concave.
(Also stated that $\mu$ is -1-convex, or AH-convex.)
A different organization of the calculation is given just after the 
calculation showing the weaker result of log-convexity.

\medskip

 A positive function $f$ is logconvex iff $f f''/(f')^2\ge{1}$, i.e. $r(f)>1$, and a short calculation shows that
 $\mu$, and, of course, also $\mu_{(2)}$ are log-convex.
Calculating we see that $\log(\mu(c))$ is convex in $c$:
 \begin{eqnarray*}
 \frac{d^2 \log(\mu(c))}{d c^2}
 &=& \frac{ \mu(c)\mu''(c)-\mu'(c)^2}{\mu(c)^2}\\
 &=&  \ \frac{(1+ \beta^2\mu^2)(3\beta^3 \mu^2 +c \beta^4 \mu^4 +2 c \beta^2 \mu^2 + \beta + c)}
 {(\beta +c (1+\beta^2\mu^2))^3} .
\end{eqnarray*}
On the basis of having calculated a few more higher derivatives  $\frac{d ^j \mu}{d  c^j}$  of $\mu(c)$ with respect to $c$
there are indications that $\mu(c)$ may be completely monotone.

Maple code related to convexity and log-convexity of $\mu$ and other forms of power convexity is given 
at the URL given in the abstract.



As stated at several places above, calculation for $\mu_{(2)}$ gives the stronger result, stronger than log-convexity:
 $$- 
\frac{2\mu_{(2)}^{5/2}}{(\mu_{(2)}')^2} \, \frac{d^2}{d c^2}\left(\frac{1}{\sqrt{ \mu_{(2)}}}\right)
 = \frac{\mu_{(2)} \mu_{(2)}''}{(\mu_{(2)}')^2} -\frac{3}{2}
= \frac{\beta^3 \mu_{(2)}}{( 1+ \beta^2\mu_{(2)})(\beta +c (1+\beta^2\mu_{(2)})}
=  - \frac{\beta^3 \mu_{(2)}'}{2( 1+ \beta^2\mu_{(2)})} ,
 $$
 a result which is re-cast in several ways, including equation~(\ref{eq:halfConc}).
 Equation~(\ref{eq:rDefProp}) gives
 $$r(\mu_{(2)})-\frac{3}{2} = \frac{1}{2}\left( r(\mu)-2 \right) .$$

A short remark on the application -- concerning fundamental eigenvalues for the Laplacian with homogeneous
Robin boundary conditions --  seems appropriate here, with full details soon in~\citeKWRob.
A nice property for Dirichlet eigenvalues is that, for any pair of convex domains,
$\lambda_1^{-1/2}$ is concave under Minkowski sums of convex domains.
For Robin eigenvalues we now have the result, for $N=1$, one dimension.
The result is that $1/\mu=1/\sqrt{\mu_{(2)}}$ is concave.
$$\frac{d (1/\mu)}{d c} = -\frac{1}{\mu^2}\, \frac{d\mu}{d c}  =
\frac{d (1/\sqrt{\mu_{(2)}})}{d c}:= -\frac{1}{2\mu_{(2)}^{3/2}}\, \frac{d\mu_{(2)}}{d c} ,
$$
so that
\begin{equation}
\frac{d^2 (1/\mu)}{d c^2}
= -\, \frac{2\beta^3 \mu (1+\beta^2\mu^2 )}{(\beta+c(1+\beta^2\mu^2 ))^3} 
=
\frac{d^2 (1/\sqrt{\mu_{(2)}})}{d c^2}
= -\, \frac{2\beta^3 \sqrt{\mu_{(2)}} (1+\beta^2\mu_{(2)} )}{(\beta+c(1+\beta^2\mu_{(2)} ))^3} .
\label{eq:halfConc}
\end{equation}
This is used in the proof of Theorem~\ref{thm:thmBorell} of~\citeKWRob. 

The preceding paragraphs concern the AH-convexity of $\mu$.
The direct calculation showing $r(\mu_{(2)})<2$, AH-concavity, is rather longer.
(We do not actually use this in Part II.)
Now, begining with the identiy noted above,
\begin{eqnarray*}
r(\mu_{(2)})
&=& \frac{1}{2}\left( r(\mu)+1\right) ,\\
&=& \frac{1}{2}\left(3 +\frac{2\beta^3 \mu_{(2)}}
{( 1+ \beta^2\mu_{(2)}) (\beta +c (1+\beta^2\mu_{(2)})} \right).
\end{eqnarray*}
We need inequalities on $\mu$.
(There are many inequalities.
An easy one, not used immediately, follows from
$\tan(x)<x$: we have $\mu>1/(c\beta)$.)
Since $\phi_1(x)\ge{1/(1+x^2)}$,  we have $1+\beta^2\mu^2\ge\beta/c$.
Using this to eliminate $c$ in the last displayed equation gives
$$ r(\mu_{(2)})
= \frac{1}{2}\left(3 +\frac{\beta^2 \mu^2}
{ 1+ \beta^2\mu^2}\right) < 2,
$$
the final inequality coming from consideration of $x^2/(1+x^2)$.

For the proof of Theorem~\ref{thm:thm3MIA}  of~\citeKWRob 
we  use $M(\ell)={\hat\mu}_{(2)}(\exp(\ell))$.
The function $M$ satisfies
$$ \sqrt{M(\ell)}\tan(\exp(\ell)\sqrt{M(\ell)}) = 1 .$$
Implicit differentiation gives
$$ \frac{d M}{d\ell}=
-\frac{2\exp(\ell) M (1+M)}{1 + \exp(\ell)(1+M)} .
$$
Differentiating again gives
\begin{equation}
\frac{d^2 M}{d\ell^2}
=  \left( -  \frac{d M}{d\ell}\right) \ 
\frac{-1+\exp(\ell) (1+3 M) +2\exp(2\ell) (1+ M)^2}{(1 + \exp(\ell)(1+M))^2} \  .
\label{eq:d2Mdell2}
\end{equation}

Maple code related to studying $M(\ell)$ is given 
at the URL given in the abstract.

\bigskip
\section*{More on higher derivatives.}
\label{pg:pghigher}

Direct calculation of the first few higher derivates of $\mu$
show that these have signs appropriate to $\mu$ being completely monotone.

If $\mu$ is completely monotong one mighte attempt to establish the signs of
the derivatives by some form of induction.

\subsection*{Direct calculation of higher derivatives of $\mu$}

Higher derivatives are calculated in maple code presented 
at the URL given in the abstract.
Where $c$ occurs in the code, it could be written in terms of $\mu$ using
$c=\beta\,\phi_1(\beta\mu)$.
Alternatively the higher derivatives can be given as functions of $\mu$ and $\mu'$
with $c$ eliminated.
In all cases expressions which are clearly one-signed arise.
However the polynomials in the numerators become unpleasantly long.

Results better than log-convexity of $\mu^{(n)}$ arise.
Hankel determinants are also studied in the code.

To date, no inductive method to establish complete monotonicity has been found.
\smallskip

\goodbreak

\subsection*{Higher derivatives of inverse functions}

The function $\phi_1$ is completely monotone,
$\phi_2$ is Stieltjes, and it is an interesting open question if their inverses,
$\mu$ and $\mu_{(2)}$ are completely monotone.

We can regard our work to date as about how second derivative information on $f$
allows us to determine second-derivative intormation on its inverse $g$.
Many questions arise concerning what can be said about higher derivatives.
In this paper we restrict ourselves to questions related to our application.
For example:\\
Give an example of a completely monotone function $f$ whose inverse $g$ is not completely monotone.\\
Under what conditions is the inverse of a Stieltjes function completely monotone?\\
Our suggestion for a method of approaching this is to use the following formula from~\cite{Jo02}:
 \begin{eqnarray}\frac{d^3 x}{d y^3}
 &=&\frac{1}{f'(x)} \frac{d}{dx}\left(\frac{d^2 x}{d y^2}\right)
 =\frac{1}{f'(x)} \frac{d}{dx}\left(-\frac{f''(x)}{f'(x)^3}\right)
 = \frac{3 f''(x)^2 - f'''(x) f'(x)}{f'(x)^5} ,\label{eq:thirdD}\\
&=&  \frac{1}{2f'(x)} \frac{d^2}{dx^2}\left(\frac{1}{f'(x)^2}\right) .
\nonumber
\end{eqnarray}

We haven't answered the questions concerning the higher derivatives, but we illustrate the
proposed technique by giving maple code we used to generate examples of functions which are
log-convex but whose inverse is not log-convex.
{\small
\begin{verbatim}
assume(xp>0);
testLogCon := u->  [simplify(subs(x = xp, u*diff(u, x$2)-diff(u, x)^2)),
         simplify(subs(x = xp, (-x*diff(u,x$2)-diff(u, x))/diff(u, x)^3))];
testLogCon(exp(-x^2)/x);
\end{verbatim}
}
Running the code gives that the function $\exp(-x^2)/x$ is not log-convex
as the first entry in the list that is returned changes sign.
However the second entry in the list is positive on $x>0$, so the
inverse function is log-convex.
Maple finds the inverse function in terms of Lambert W functions.

To find an example of a completely monotone function whose inverse is not completely monotone
one might feed completely monotone functions into code like that above,
outputting, also, the 3rd derivative of the inverse using equation~(\ref{eq:thirdD}).

\subsection*{Higher derivatives of quotients}

Establishing that $\mu(c)$ is completely monotone is equivalent to establishing that
$-\mu'$ is.
We have
$$ -\mu' 
= \frac{u}{v} , $$
with
$$u= \mu(1+\beta^2\mu^2) ,\qquad
v= \beta(1+\phi_1(\beta\mu)(1+\beta^2\mu^2)) .$$

This leads to considering the $n$=th derivative of a quotient:
$$\frac{d^n}{dx^n}\left(\frac{u(x)}{v(x)}\right)
=\frac{d^n}{dx^n}\left(u(x)\cdot \frac{1}{v(x)}\right)
=\frac{d^n}{dx^n}\left(u(x)\cdot v(x)^{-1}\right) .
$$
By the general product rule, this is equal to
$$\sum_{k=0}^n\binom{n}{k} \frac{d^{n-k}}{dx^{n-k}} u(x)\cdot\frac{d^k}{dx^k}v(x)^{-1} .
$$
So the problem reduces to finding the $k$-th derivative of $v(x)^{-1}$. 
This can be done by applying Faà di Bruno's formula for the $n$-th derivative of a composition of functions.

A reference related to this is\\
Robert A. Leslie,\\
How Not to Repeatedly Differentiate a Reciprocal,\\
{\it The American Mathematical Monthly},
{\bf 98}(8) (Oct., 1991), 732-735.\\
DOI: 10.2307/2324425\\
URL: \verb$http://www.jstor.org/stable/2324425$

\medskip
A source of some of the references and formulae is:\\
{\small
\verb$https://math.stackexchange.com/questions/5357/whats-the-generalisation-of-the-quotient-rule-for-higher-derivatives?rq=1$
}

A recurrence is
$$\frac{d^n}{dx^n}\left(\frac{u(x)}{v(x)}\right)
=\frac{1}{v(x)}\left(u^{(n)}(x)-
n!\sum_{j=1}^n \frac{v^{(n+1-j)}(x)}{(n+1-j)!}
\frac{ (\frac{u(x)}{v(x)})^{(j-1)}}{(j-1)!}\right) .
$$
See:\\
Christos Xenophontos,\\
A formula for the $n$-th derivative of the quotient of two functions, (2007)
\medskip

\noindent
Related papers:\\
F Gerrish ,\\
A useless formula?
{\it Mathematical Gazette}, 1980 - cambridge.org\\
\smallskip
P Shieh, K Verghese,\\
A general formula for the nth derivative of $1/f(x)$,
{\it American Mathematical Monthly}, 1967

\newpage


\section*{Completely monotone functions and related topics}


This appendix has general facts concerning
completely monotone and absolutely monotonic functions and subsets thereof that
might be of relevance to our study of $\mu(c)$.


It may be that $\mu$ and hence $\mu_{(2)}$ are completely monotone.
However we note below that 
GA-convexity 
is not satisfied by every completely monotone function.

\subsection*{Log-convex functions}

\par\noindent
{\bf Lemma LC1.} {\it The function $f$ is log-convex on an interval $I$, if and only if for all 
$a, b, c \in{ I}$ with $a < b < c$, the following holds:
$$f(b)^{c-a} \le  f(a)^{c-b} \le f(c)^{b-a} . $$
}
\medskip
Hence for $r\le{1}\le{1/r}$,
$$ 1 \le \left( \frac{\mu(r)}{\mu(1)}\right)^{\frac{1}{r}-1} \left( \frac{\mu(\frac{1}{r})}{\mu(1)}\right)^{1-r} . $$

\vspace{1cm}

\subsection*{Completely monotone functions}

\par\noindent {\bf Definition.}
A real-valued function $f$ defined on $[0,\infty)$ is said to be
{\it completely monotone} (totally monotone, completely monotonic, totally monotonic)
  if   
 $(-1)^k f^{(k)}(x)\ge{0}$ for $x>0$ and $k=0,1,2,\ldots$.\\
Denote the set of completely monotonic functions on $[0,\infty)$ by ${\CM}$.
\medskip

\par\noindent
{\bf Theorem CM1.} {\it
The set ${\CM}$ forms a convex cone: $(t_1 f_1 + t_2 f_2)\in{\CM}$
for all nonnegative numbers $s$, $t$ and all $f_1\in{\CM}$ and $f_1\in{\CM}$.\\
The set ${\CM}$ is also closed under multiplication and point-wise convergence. That is
$$f_1(x) f_2(x) \in {\CM}\qquad{\rm	and\ \ }
\lim_{n\rightarrow\infty} f_n(x) \in{\CM}, $$
where $f_n(x)\in{\CM}$ for all $n\ge{1}$ and their point-wise limit exists for any $x > 0$.}
\medskip

\par\noindent
{\bf Theorem CM2.} {\it Let $f(x)\in{\CM}$. and let $h(x)$ be nonnegative with its derivative in ${\CM}$.
Then $f(h(x)\in{\CM}$.}
\medskip

\par\noindent
{\bf Corollary CM2.} {\it Let $f(x)\in{\CM}$ and $f(0) < \infty$. Then the function
$$ -\log\left( 1-\frac{f(x)}{A}\right) . \qquad
A\ge{f(0)}, 
$$
is ${\CM}$. From this it follows that
$$\frac{f'(x)}{A-f(x)}. \qquad A\ge f(0) $$
is ${\CM}$ since this reduces to minus the derivative of the previous expression.
}
\smallskip

\par\noindent This corollary is given in~\cite{MS01}.

\medskip

From the derivative of the last function, we have
$$\frac{d}{d x}  \frac{f'(x)}{A-f(x)}
= \frac{ f'' (A-f) + (f')^2}{(A-f)^2} \le {0} . $$
Rearranging this gives
$$ f f'' - (f')^2\ge A f'' \ge{0} . $$
In particular, any $f\in\CM$ is log-convex.
\smallskip

There is a relationship of $\CM$ functions and Laplace transforms.
Define
$$F(r)= \int_0^\infty \exp(-r t) f(t) \, dt .$$
If the above integral is bounded for all $r>0$ and $f(t)\ge{0}$ for all $t>0$
then $F\in\CM$.
The converse is also true.
\smallskip

That the Laplace transform of a positive function is log-convex can be proved using the
Cauchy-Schwarz inequality.  Suppose $f(t)\ge{0}$. Then
\begin{eqnarray*}
F(\frac{r_1+r_2}{2})
&=&\int_0^\infty \exp(-\frac{r_1+r_2}{2}\, t) f(t)\, dt \\
&=&  \int_0^\infty \exp(-\frac{r_1}{2}\, t) \sqrt{f(t)}\,  \exp(-\frac{r_2}{2}\, t) \sqrt{f(t)} \, dt
\le \sqrt{F(r_1)\, F(r_2)} .
\end{eqnarray*}
See \S4.8 of~\cite{Wa92}.

\medskip

If we were to attempt to establish 
GA-convexity 
via Laplace transforms of positive functions, one
might begin with
$$ F(r)+F(\frac{1}{r})-2 F(1)
= \int_0^\infty k(r,t) f(t)\, dt ,
$$
where
$$ k(r,t)=\exp(-r t) + \exp(-\frac{t}{r}) - 2\exp(-t) . $$
Also define
$$ k_c(r,t)=\exp(-r t) + \exp(-\frac{t}{r}) - 2\exp(-\frac{1}{2}(r+\frac{1}{r})\, t) \ (\ge{k(r,t)})\ . $$
Now $k_c(r,t)\ge{0}$ for $r>0$ and $t>0$ and this is another way to show $F$ is convex.
It happens that, for every $r>0$,  $k(r,t)$ takes on both signs, and, as a consequence, there are functions $F\in\CM$
for which the Laplace representation does not yield GA-convexity. 
The fact that $\mu_{(2)}$ does is as a consequence of it being the inverse of the
completely monotone (indeed Stieltjes) function $\phi_2$.
At fixed $r>0$, $k(r,t)<0$ for $0<t<1/2$ and $k(r,t)>0$ for $t>1$, and there is a unique $t_{0}(r)$ in
the interval $(1/2,1)$ where $k(r,t_{0}(r))=0$.

\medskip

Another published statement, a special case of which is that
any $f\in\CM$ is log-convex is as follows:\\
{\it Let $f\in{\CM}$ . Then
$$ (-1)^{nk}\left( f^{(k)}(x)\right)^n \le (-1)^{nk} \left( f^{(n)}(x)\right)^k (f(x))^{n-k} $$
for all $x>0$ and integers $n\ge{k}\ge{0}$.}\\
In particular, for $n=2$ and $k=1$ we have that any completely monotonic function is log-convex.

\subsection*{Absolutely monotonic  functions}

A function $f(x)$ is {\it absolutely monotonic} in the interval $a<x<b$ if it has nonnegative derivatives of all orders in the region, i.e.,
$f^{(k)}(x)\ge{0}$ for all $x$ in the interval and $k=0,1,2,\ldots$\ . 
Denote by $\AM{(a,b)}$ the set of all functions {absolutely monotonic} in the interval $a<x<b$.
\medskip

\par\noindent
{\bf Theorem AM1.} {\it
The set ${\AM}{(a,b)}$ forms a convex cone: $(t_1 f_1 + t_2 f_2)\in{\AM}{(a,b)}$
for all nonnegative numbers $s$, $t$ and all $f_1\in{\AM}$ and $f_2\in{\AM}{(a,b)}$.\\
The set ${\AM}$ is also closed under multiplication and point-wise convergence. That is
$$f_1(x) f_2(x) \in {\AM}{(a,b)}\qquad{\rm	and\ \ }
\lim_{n\rightarrow\infty} f_n(x) \in{\AM}{(a,b)}, $$
where $f_n(x)\in{\AM}{(a,b)}$ for all $n\ge{1}$ and their point-wise limit exists for any $x > 0$.}
\medskip

In particular the function $X\tan(X)$ is absolutely monotonic on $[0,\pi/2)$.
\medskip

Widder (1941)~\cite{Wi41} gives:

\par\noindent
{\bf Theorem AM2.} {\it
$f\in\AM$ and $g\in\CM$ then the composition $f\circ{g}\in\CM$.
}

\subsection*{Miscellaneous topics}

Further classes of functions, Stieltjes functions, Bernstein functions, etc., are treated in~\cite{SSV}.

Let {\bf CMI} tbe the set of functions $\phi$ with domain and range $(0,\infty)$, 
with  $\phi\in\CM$, and for which the inverse $\phi^{-1}$ is also in $\CM$.
The set {\bf CMI} is nonempty as
$f(\alpha,z)=z^{-\alpha}$ is in $\CM$ when $\alpha>0$ and
$f(1/2,f(2,z))=z$ so $f(1/2,\cdot)$ has as its inverse $f(2,\cdot)$.
The question arises as to whether either of the functions $\phi_1$ or $\phi_2$ 
is in {\bf CMI}.
There are two questions related to this.\\
(i) If the $\phi$ are in {\bf CMI} how might this help with estimates related to the Robin eigenvalue
treated in \citeKWRob.\\
(ii) Under what conditions is the inverse of a function $f\in\CM$ also in $\CM$? And, do our $\phi$ satisfy these conditions?

\medskip

\subsubsection*{Further comments on Stieltjes functions}

(The next paragraph is based on\\
Stieltjes functions of finite order and hyperbolic monotonicity\\
Lennart Bondesson       and     Thomas Simon\\
arXiv:1604.05267v1\\
but most of that paper isn't relevant here.)

It was shown by Widder - see Theorem 1 of~\cite{So10} --  that 
$f\in{\cal S}$ if and only if $f$ is smooth and
$$(x^n f)^{(n)} \in\CM\qquad \forall n\ge{0} .$$
Another theorem 
Theorem 18b p. 366 in~\cite{Wi41} -
states that a non-negative function $f$ is in ${\cal S}$ if and only if it is smooth and such that
$$(-1)^{n-1} (x^n f)^{(2n-1)}   \ge{0}  \ \forall n\ge 1.$$

\medskip

Any nonzero Stieltjes function $\phi$ is logarithmically completely convex,
the definition of which is
$$ f {\mbox{\rm\ \  is\ \  log-}}\CM
\Longleftrightarrow (-1)^n \frac{d^n}{d x^n} \log(f(x)) \ge{0} . $$
C. Berg, Integral representation of some functions related to the gamma function, 
{\it Mediterr. J. Math.} {\bf 1}(4) (2004),  433Ð439;\\ 
\verb$http://dx.doi.org/10.1007/s00009-004-0022-6$\\
B.-N. Guo and F. Qi, 
A completely monotonic function involv- ing the tri-gamma function and with degree one, 
{\it Appl. Math. Comput.} {\bf 218}(19) (2012), 9890Ð9897;\\
\verb$http://dx.doi.org/10.1016/j.amc.2012.03.075$\\
F. Qi and C.-P. Chen, 
A complete monotonicity property of the gamma function, 
{\it Math. Anal. Appl.} {\bf  296} (2004), 603Ð607; \\
\verb$http://dx.doi.org/10.1016/j.jmaa.2004.04.026$\\
F. Qi and B.-N. Guo, 
Complete monotonicities of functions involving the gamma and digamma functions, 
{\it RGMIA Res. Rep. Coll.} {\bf 7}(1) (2004),  63Ð72;\\
\verb$http://rgmia.org/v7n1.php$

\section*{Decreasing convex involutions} 

Define ${\rm invol}_0(x)=\frac{1}{x}$, which is a decreasing convex involution.
The only involution which is Stieltjes is $1/x$: see Fact~\ref{fac:Fact3}.
A question we have asked ourselves, as yet unanswered is:\\
Might there be other involutions, invol for which,
with $f$ and $g$ both positive decreasing convex functions, $g$ being the inverse of $f$,
${\rm invol}\circ(g)$ concave implies  $f\circ{\rm invol}$ is convex?

As an example of another decreasing convex involution on $(0,\infty)$ we instance
$${\rm invol}_1(x)=\ln \left({\frac {e^{x}+1}{e^{x}-1}}\right)
=\ln ( \coth(\frac{x}{2}))
= 2\, {\rm arctanh}(\exp(-x))
= 2\, {\rm arccoth}(\exp( x)). $$
Note
$${\rm arccoth}(z)= \frac{1}{2}\log\left(\frac{z+1}{z-1}\right)
\ {\rm for\ }\ z^2>1,\qquad
{\rm arctanh}(z)= \frac{1}{2}\log\left(\frac{1+z}{1-z}\right)
\ {\rm for\ }\ 0<z^2<1 . $$
One can show ${\rm invol}_1\in\CM$.
It appears that ${\rm invol}_1$ is GG-concave.

\medskip
As another example of another decreasing convex involution on $(0,\infty)$ we instance
$${\rm invol}_2(x)=\exp\left(\frac{1}{\log(1+x)}\right) - 1. $$
Again it appears that ${\rm invol}_2\in\CM$. Formally, it is
$$\exp\left(\frac{1}{\log(1+x)}\right) - 1
= \sum_{k=1}^\infty \frac{1}{k! (\log(1+x))^k} $$
which is $\CM$ as $1/\log(1+x)$ is.
Actually $1/\log(1+x)\in{\cal S}$ as in~\cite{SSV} item 26, p228
(but, as the product of functions in $\cal S$ need not be in $\cal S$ the
extra information does not appear useful).
It appears that ${\rm invol}_2$ is GG-convex.

\medskip
We remark that neither ${\rm invol}_1$ nor ${\rm invol}_2$ is Stieltjes.
As stated above, the only involution which is Stieltjes is $1/x$: see Fact~\ref{fac:Fact3}.
If one has $s\in{\cal S}$ then its reciprocal $(1/s)\in\CBF$,
hence Bernstein, hence increasing, concave.
However the reciprocals of ${\rm invol}_1$ and of ${\rm invol}_2$ --
both increasing functions, of course --
are neither convex, nor concave.

\newpage
\section*{Some complex variable}
\label{pg:pgcomplex}

Another approach to extablishing $\mu$ to be completely monotone (if, as we guess, it is)
is via complex variable methods.
There are two separate strands to our efforts.
\begin{itemize}

\item
The first is a development from the famous paper~Burniston and Siewert~\cite{BS73}.
It may be that the integral representation of $\mu$ may yield properties of $\mu$.

\item
The second, less well developed, involves investigations related to Pick functions
in general, and functions involving square-roots, tangents and arctangents
in particular. 
(The function denoted $X(w)$ in the following may be worthy of study.)

\end{itemize}

\subsection*{A formula from Burniston and Siewert~\cite{BS73}}

The implicit equation for $\mu(c)$,
$$ \mu \tan(c\mu)= \frac{1}{\beta} ,$$
can be recast with $X=c\mu$ and $w=c/\beta$ as
$$ X\tan(X) = w .\eqno{(A)}$$
(The variable here denoted $w$ is, elsewhere in our papers, denoted $\hat c$,
and in~\cite{BS73} is denoted by~$\omega$.)

Define, as in~\cite{BS73} equation (2.32),
$$\Lambda_0(z) 
= z\left(z-\frac{1}{2 w}\log\left(\frac{z-1}{z+1}\right)\right) . $$
The smallest postive root, when $w>0$, is given by
$$ X(w) = \sqrt{\frac{\pi w}{2}}\,
\exp\left(-\,\frac{1}{\pi}\,
\int_0^1 \left({\rm arg}(\Lambda_0^+(t))+\frac{\pi}{2}\right)\, \frac{dt}{t} \right) . $$

Manipulating this suggested
$$ X(w) = \sqrt{\frac{\pi w}{2}}\,
\exp\left(-\,\frac{1}{\pi}\,
\int_0^1 \arctan\left(\frac{
\log\left(\frac{1+t}{1-t}\right) + 2 t w}{\pi}\right) \, \frac{dt}{t} \right) .
$$
The log term in the integrand can be written in terms of the inverse hyperbolic tangent function:
$$ \log\left(\frac{1+t}{1-t}\right) = 2\,  {\rm arctanh}(t). $$
The definite integral can be evaluated at a few particular values of $w$.
For example, at $w=0$ we get
$$\exp\left(-\,\frac{1}{\pi}\,
\int_0^1 \arctan\left(\frac{
\log\left(\frac{1+t}{1-t}\right) }{\pi}\right) \, \frac{dt}{t} \right) 
=\sqrt{\frac{2}{\pi}} , $$
and this is consistent with $X(w)\sim\sqrt{w}$ as $w\rightarrow{0}$.

\medskip

There are a number of tasks which should be attempted.\\
Check the integral above  satisifies the d.e. known to be satisfied by the 
$X(w)$ satisfying the implicit relation given at equation (A).\\
Perhaps $X(w)$ is Bernstein. (If this is the case, then $\mu(c)$ is completely monotone.)\\
Perhaps, for $c_0\ge{0}$ and $c_1>0$, $\arctan(c_0 +c_1 \sqrt{x})$ is Bernstein in $x$.\\
If the line above is true then it may have implications for $X$.

Maple code for my variant of the Burniston-Siewert representation for $X(w)$ follows.

{\small
\begin{verbatim}
solInt := proc (w) 
 local tr; 
 evalf(sqrt((1/2)*Pi)*exp((1/2)*log(w)-
    evalf(Int(arctan((log((tr+1)/(-tr+1))+2*tr*w)/Pi)/tr, tr = 0 .. 1))/Pi)) 
end proc:

solInt(0.5);
fsolve(X*tan(X) = 0.5, X, 0 .. 2); # both give about 0.65327
\end{verbatim}
}

\subsection*{Properties of square roots and arctan}

From~\cite{AS65}: 
$$\leqno{(4.3.57)} 
 \tan(x+i y)
=\frac{\sin(2x)+i\sinh(2y)}{\cos(2x)+\cosh(2y)} ,
$$
So $\tan(z)$ takes the upper half-plane into the upper half-plane.
$$\leqno{(4.4.39)}
 \arctan(x+i y)
=\frac{1}{2}\arctan\left(\frac{2x}{1-x^2-y^2}\right) +
\frac{i}{4} \log\left(\frac{x^2+(y+1)^2}{x^2+(y-1)^2}\right) .
$$
So $\arctan(z)$ takes the upper half-plane into the upper half-plane.

Using the branches as in Maple14:
\begin{itemize}

\item $\arctan(1/z)$ takes the upper half-plane to the lower half-plane.

\item $\arctan(1/\sqrt{z})$ takes the upper half-plane to the lower half-plane.

\item $1/\sqrt{z}$ takes the upper half=plane to the lower half-plane.

\end{itemize}

\subsection*{Use of Cauchy Residue Theorem for solving nonlinear equations}

Suppose $S(z)$ has a simple pole at $z_0$.
Cauchy's Theorem gives, for a closed path $\gamma$ sufficiently close to $z_0$,
$$ \int_\gamma (z-z_0) S(z)\, dz = 0 . $$
Hence
$$ z_0 = \frac{\int_\gamma z S(z)\, dz}{\int_\gamma S(z)\, dz} . $$
This is applied to our nonlinear equation in~\cite{LS02}, with
$$ S(z)= \frac{1}{z\sin(z)-c\cos(z)} . $$

\subsection*{Another approach to inverses of functions of a complex variable}

Let $f:A\mapsto{B}$ be analytic, one-to-one and onto.
Let $w\in{B}$ and let $\gamma$ be a small circle centred at $z_0$ in $A$.
Then the inverse to $f$ is given by the formula
$$ f^{-1}(w) = \frac{1}{2\pi i}\int_\gamma 
\frac{f'(z) z}{f(z)-w}\, dz ,$$
for $w$ sufficiently close to $f(z_0)$.
(See Marsden {\it Basic Complex Analysis} p397 Q14.)

\subsection*{Pick functions}

From~\cite{SSV}. An analytic function that preserves the upper half-plane is called
a {\it Pick function} (or a {\it Nevanlinna function} or a {\it Nevanlinna-Pick function}).\\
In other words, a Pick
 function is a complex function which is an analytic function on the open upper half-plane 
and has non-negative imaginary part.\\
 A Pick function maps the upper half-plane to itself or to a real constant, 
but is not necessarily injective or surjective. 
\medskip

Clearly if $M(z)$ and $N(z)$ are Nevanlinna-Pick functions, 
then the composition $M(N(z))$ is a Nevanlinna-Pick function as well.
\medskip

Examples of Nevanlinna-Pick functions:

\begin{itemize}

\item $\sqrt{z}$. \cite{AS65} (3.7.27) gives, with $r-\sqrt{x^2+y^2}$,
$$\sqrt{x+ iy}
= \sqrt{\frac{1}{2}(r+x)}+ i\, {\rm sign}(y)\, \sqrt{\frac{1}{2}(r-x)} ,
$$

\item $-1/\sqrt{z}$ 

\item $\tan(z)$ (an example that is surjective but not injective)

\end{itemize}

\medskip
\noindent
{\sc Theorem.~\cite{SSV}}{\it Suppose that $b$ is Bernstein
(so, in particular a non-negative function on $(0,\infty)$).
Then\\ 
{\rm Cor 3.7(iv)} $b(w)/w$ is in $\CM$.\\
{\rm Thm 3.6} $\exp(-s b(w))$ is in $\CM$ for any $s>0$.}

\medskip
\noindent
{\sc Theorem.~\cite{SSV} 6.2.}{\it Suppose that $b$ is a non-negative function on $(0,\infty)$.
Then the following conditions are equivalent.\\
(i) $b\in\CBF$.\\
(ii) The function $z\mapsto b(z)/z$ is in $\cal{S}$.\\
(iv) $b$ has an analytic continuation to the upper half-plane such that 
{\rm Im}$(b(z))>0$ for all $z$ in the open upper half-plane and
such that the limit, $b(0+)=\lim_{(0,\infty)\ni{z}\rightarrow{0}} b(z)$, exists and is real (nonnegative).\\
(v) $b$ has an analytic continuation to the cut complex plane $C\setminus(0,\infty)$
such that ${\rm Im}(z){\rm Im}(b(z))>0$ and 
such that the limit, $b(0+)=\lim_{(0,\infty)\ni{z}\rightarrow{0}} b(z)$, exists and is real (nonnegative).}

\medskip
When $b(z)=z\,\phi_2(z)$, $b(0+)=0$.

\section*{Miscellaneous associated with I.\S\ref{sec:transcendental} and II.\S\ref{sec:explicit}}

\noindent
O. F. de Alcantara Bonfima David J. Griffiths,
Exact and approximate energy spectrum for the finite square well and related potentials,
{\it Amer J. Physics} {\bf 74} (1) (2006) p43 .\\
\verb$DOI: http://dx.doi.org.dbgw.lis.curtin.edu.au/10.1119/1.2140771$

\noindent
V. Barsan,
Algebraic approximations for transcendental equations with applications in nanophysics,
{\it Philosophical Magazine}, September 2015\\
{\bf 95}, No. 27, 3023Ð3038\\
\verb$http://dx.doi.org/10.1080/14786435.2015.1081425$

\section*{Formerly with I.\S\ref{sec:transcendental} and II.\S\ref{sec:Oned}}

\noindent{\bf In the subsection on properties of $\phi$}

$\phi_2$ is better than logconvex.
$1/\phi_2$ is concave.

Maple code treating this topic is presented at the URL given in the abstract.
\bigskip

\subsection*{Further properties of log-convex functions, etc.}

Learnt from:\\
Tomislav Doslic, Log-convexity of combinatorial sequences from their convexity,
{\it Journal of Mathematical Inequalities}  September 2009 DOI: 10.7153/jmi-03-43
A positive function $f(x)$ is logconvex iff the function $\exp(\alpha x)f(x)$ is convex for all real $\alpha$.\\
(From this one can get a simple proof that the sum of logconvex functions is logconvex.)\\
D. V. Anosov, On the sum of log-convex functions, 
{\it Math. Prosv.}, {\bf 5} (2002), 158Ð163 (in Russian).

\begin{itemize}
\item
Further possible sources of information on logconvex functions, etc, is:\\
J.E.Pecaric, F.Proschan AND Y.L.Tong,
{\it Convex Functions, Partial Orderings and Statistical Applications}, Academic Press, Boston, 1992.

\item M. Avriel {\it Generalized concavity} (1987)\\

\end{itemize}
\goodbreak

\newpage

\begin{center}
{\huge{\textbf{\textsc{ Part II: Inequalities for the fundamental Robin eigenvalue for the Laplacian on
$N$-dimensional rectangular parallelepipeds
}}}}
\end{center}

\bigskip
\section*{Abstract}
\label{pg:absII}
Amongst $N$-dimenstional rectangular parallelepipeds ({\boxshape}es) of a given volume,
that  which has the smallest fundamental Robin eigenvalue for the Laplacian is the $N$-cube.
We give an elementary proof of this isoperimetric inequality based on the well-known formulae for the eigenvalues.
Also treated are various related inequalities which are amenable to investigation using the formulae
for the eigenvalues.

\section{Introduction}\label{sec:rIntro}

\subsection{Overview}\label{subsec:Overview}

This paper has its origins in earlier papers by the authors.
In the more recent of these, currently in a pre-publication form intended to supplement a shorter paper, ~\cite{KW16s},
 the application required $N=2$ and Theorem 3 of that paper is the $N=2$ case of
Theorem~\ref{thm:thm3ArXivV2} in this paper.
The main physical application in the older paper~\cite{MK94}, heat flow,  required $N=3$ but the paper often considered general values of $N\ge{2}$.
In the general setting, $\Omega$ is a bounded simply-connected domain in $R^N$, with piecewise $C^1$ boundary.
We will soon study the special case when $\Omega$ is a rectangular parallelepiped, here called a \boxshape.
(Other words for the same shape include rectangular cuboid, hyper-rectangle, and $N$-orthotope.)
The partial differential equation problems of~\cite{MK94} can be written, with $\torsforce\ge{0}$ and $\beta\ge{0}$ given, as
\begin{equation}
\Delta u +\lambda u= -\torsforce \qquad{\rm in\ }\ \Omega, \qquad
\beta\frac{\partial u}{\partial n}+ u= 0\  {\rm\ on\ }\partial\Omega \ .
\label{eq:gen}
\end{equation}
Here $n$ is the outward normal, and the second equation is called a `Robin boundary condition'.
(In the context of fluid flows  with slip at the boundary it is called Navier's boundary condion.
In elasticity, it arises with boundaries that are elastically supported, see e.g.~\cite{Sp03}.
In the context of heat diffusion it is sometimes called Newton's law of cooling.
See~\cite{GA98}.)
The case $\lambda=0$ and $\torsforce=1$ is Problem $P(\beta)$ of~\cite{MK94}.
The case $\torsforce$=0 is Problem $H(\beta)$ of~\cite{MK94}, and it is this Helmholtz problem that is the subject of this paper.
Problem $H(\beta)$ is a
Sturm-Liouville style eigenvalue problem where one seeks solutions $(\lambda_k, u_k)$, with $ u_k$ not identically zero, satisfying
\begin{equation}
\sum_{j=1}^N\frac{\partial^2  u_k}{\partial x_j^2} \
+\lambda_k  u_k = 0\ {\rm\ in\ }\ \Omega,\qquad
\beta\frac{\partial u_k}{\partial n}+ u_k = 0\  {\rm\ on\ }\partial\Omega \ .
\label{eq:Helmholtz}
\end{equation}
The $\lambda_k$ are all positive and  ordered so that $\lambda_{k+1}\ge{\lambda_k}$.
Completeness of the eigenfunctions and expansions in terms of them is
treated in classical mathematical methods books such as~\cite{CHI57}.
When $k=1$, `the fundamental', the eigenfunction is one-signed, and we take it to be positive.

This note concerns the fundamental eigenvalue, that is $k=1$, and $\Omega$ is
the \boxshape
\begin{displaymath}\Omega=(-a_1,a_1)\times(-a_2,a_2)\times \ldots \times(-a_N,a_N) .\end{displaymath}
We sometimes indicate the dependence on the list of $a_j$ by writing $\Omega({\mathbf a})$.
Some of our results, e.g. Theorem~\ref{thm:thm3ArXivV2},
(but not Theorems~\ref{thm:thmMonot},~\ref{thm:thm3MIA},~\ref{thm:thmScale} or~\ref{thm:thmBorell})
require the volume to be fixed.
Then we consider the family of {\boxshape}es, with volume $(2h)^N$, with $a_j=r_j h$ with all $r_j>0$ and
the product of the $r_j$ equal to 1.
We denote the volume of $\Omega$ by $|\Omega|$, so that
\begin{equation}
|\Omega({\mathbf a})|=2^N\prod_{j=1}^N a_j = (2 h)^N .
\label{eq:vola}
\end{equation}
Our goal throughout the paper is to present results which are valid for $\beta\ge{0}$, and it is usually the
situation that the $\beta=0$ case is well known.
Of our results the easiest to state is the following isoperimetric inequality.

 \begin{theorem}\label{thm:thm3ArXivV2}
Amongst all {\boxshape}es with a given volume, that which has the smallest fundamental Robin eigenvalue
is the cube.
\end{theorem}

There is a geometric isoperimetric inequality relating volume and perimeter for \boxshape{es},  namely that there is
a $C_\square(N)>0$ (whose numeric value is unimportant here) such that
$$C_\square(N) |\Omega|^{N-1}\le |\partial\Omega|^N ,\qquad
{\rm and\ \ } C_\square(2)=16, \ \ C_\square(3)=36 .
$$
This combines with the domain monotonicity for  {\boxshape}es, Theorem~\ref{thm:thmMonot}, to yield the following Corollary.
With $\beta>0$ and $N=2$ we will see
(in \S\ref{subsec:Othergeom}) another proof as a corollary of Theorem~\ref{thm:thmBorell}.

\begin{corollary}\label{cor:corPerim}
Amongst all \boxshape{es} with a given perimeter, that which has the smallest eigenvalue is the cube.
\end{corollary}

Our proofs of Theorem~\ref{thm:thm3ArXivV2} when $\beta>0$
(the simplest, given in~\S\ref{subsec:thm3}, being a consequence of Lemmas~\ref{lem:rlem1} and~\ref{lem:lem2})
are elementary.
However our search of the literature has failed to find the result, even with other proofs
(and we know that many alternative proofs are possible).
We are aware that results can be hard to locate in older literature.
Indeed in~\S{3} of~\cite{MK94} we unwittingly rediscovered the formulae for the fundamental Robin eigenfunction
for the equilateral triangle, a century and a half after Lam{\'e},
and only learnt that it was a rediscovery a decade after our paper.
Lam{\'e}'s work is referenced in this journal in~\cite{LPP12}.

When $\beta=0$, the result of Theorem~\ref{thm:thm3ArXivV2} is well-known, and very elementary.
Then one has the volume fixed $|\Omega|=2^N\prod_{j=1}^N a_j$ and minimizes
the fundamental {\it Dirichlet} eigenvalue
\begin{equation}\lambda_1(\beta=0)
= \frac{\pi^2}{4} \sum_{j=1}^N \frac{1}{a_j^2} .
\label{eq:lambda1beta0}
\end{equation}
Fixing the volume is equivalent to fixing the product $\prod_{j=1}^N {1}/{a_j^2}$,
and with this observation, we see that the $\beta=0$ case is equivalent to the
equality case of the AM$\le${GM} inequality.
(See~\cite{HLP34} \S2.5 p17.)
AM abbreviates `Arithmetic Mean', GM, `Geometric Mean'.
The  minimization problem for $\lambda_1$, when $\beta=0$,
is similar to that in which one seeks, instead, to minimize the perimeter
(again with a constant $C(N)$ which is unimportant here)
\begin{displaymath}|\partial\Omega|
= C(N) |\Omega| \sum_{j=1}^N \frac{1}{a_j} .\end{displaymath}

Return now to the situation where $\beta\ge{0}$.
The fundamental eigenvalue is given by a formula 
$$\lambda_1(\Omega({\mathbf a}))=\sum_{j=1}^N \mu({a_j})^2 .$$
The function $\mu$,
as given in~\S\ref{sec:explicit} equation~(\ref{eq:muXmuY}),
also depends on $\beta$, but in contexts where $\beta$ is fixed we omit it.
When the value of $\beta$ needs to be indicated, we write $\mu(\beta,a)$ as, for example,
$\mu(0,a)=\pi/(2 a)$.
Our proofs of
the isoperimetric result of Theorem~\ref{thm:thm3ArXivV2}
follow from properties of $\mu$, or of its inverse.
The inverse of $\mu$ is an elementary function and this leads to simple proofs.
The fact that $\mu(a)$ is convex, indeed log-convex, is, by itself,
insufficient to establish that the separable convex optimization problem,
minimize $\lambda_1(\Omega({\mathbf a}))$
subject to the volume constraint~(\ref{eq:vola}),
has as its solution that all $a_j$ are equal with $2a_j=|\Omega|^{1/N}$ for all $j$.
For the proof of Theorem~\ref{thm:thm3ArXivV2} additional properties of $\mu$ -- or of its inverse -- are needed.
What is actually needed is clear from the following restatement:

\smallskip
{\sc Theorem~\ref{thm:thm3ArXivV2} Restated.}\ {\it The function $\mu_{(2)}$, where $\mu_{(2)}(c)=\mu(c)^2$, is
GA-convex.}
\smallskip
The equation expressing this is that of~(\ref{eq:gGA}), with the function $g$ there being replaces by $\mu_{(2)}$.
For the definition of GA-convex, and other forms of generalised convexity used in this paper,
see~\citeKeconv.
The notation $\mu_{(2)}$ is introduced as, amongst other reasons,
 differentiation with respect to it is easier to read than
attempting to use $\mu^2$.

\medskip

An outline of our paper is as follows.
In \S\ref{subsec:Othergeom} we note some other results which are familiar in the case $\beta=0$, and
which extend to our \boxshape{es} with $\beta>0$.
For domains more generally much more has been established when $\beta=0$,
and there are proof techniques which are not available for $\beta>0$,
e.g. that which we review in~\S\ref{subsec:Steiner}.


The main subject of this paper is the \boxshape\ domain, and we return to this in
all subsequent sections.
In~\S\ref{sec:explicit} we give the exact solution for the eigenfunction,
and the implicit equation for its eigenvalue.
This is developed in~\S\ref{sec:Oned}.
These combine with a lemma from~\S\ref{subsec:SeparableC} to yield,
in~\S\ref{subsec:thm3}, our neatest proof of Theorem~\ref{thm:thm3ArXivV2}.
Further inequalities on the fundamental Robin eigenvalue
for \boxshape{es} we treat are summarised immediately below in \S\ref{subsec:Othergeom}.
Their proofs are in~\S\ref{sec:stated1p1}.
In the discussion in~\S\ref{sec:Discussion} we mention some open questions, both for the eigenvalue problem,
$H(\beta)$,
and the generalized torsion problem $P(\beta)$.

For domains more generally, much is known about the fundamental Robin eigenvalue.
The corresponding eigenfunction can be taken to be positive, and will be in this paper.
The generalization of the Faber-Krahn inequality is given in~\cite{Da06}.
In that context one has an isoperimetric result associated with varying over all bounded domains with a given volume,
with the $N$-dimensional ball as the optimizer.
In this paper we vary over {\boxshape}es with a given volume, again obtaining,
in Theorem~\ref{thm:thm3ArXivV2}, an `isoperimetric'  inequality.

\subsection{Other ways the geometry of \boxshape{es} affects $\lambda_1$}\label{subsec:Othergeom}

In dimensions $N\ge{2}$, general domain monotonicity, true for $\beta=0$,  isn't true when $\beta>0$.
Rectangles, \boxshape{es},  however, behave nicely, as noted in the following.

\begin{theorem}\label{thm:thmMonot}
When $\beta\ge{0}$, \boxshape{es} inherit domain monotonicity from the domain monotonicity that is
present when $N=1$ as given by the formula~(\ref{eq:lambda1Rect}).
That is, if $\Omega_1\subseteq\Omega_2$ then $\lambda_1$ satisfies the inequality
$\lambda_1(\Omega_1)\ge\lambda_1(\Omega_2)$.
\end{theorem}

The same GA-convexity of $\mu_{(2)}$ as in Theorem~\ref{thm:thm3ArXivV2} Restated leads to the following:

 \begin{theorem}\label{thm:thm3MIA}
The fundamental eigenvalue $\lambda_1(\Omega({\mathbf a}))=\sum_{j=1}^N \mu({a_j})^2$
as given in~\S\ref{sec:explicit} equation~(\ref{eq:lambda1Rect}),  is a convex function of the $\ell_j=\log({a_j})$.\\
Equivalently, with ${\mathbf a}(t)=(a_j(0)^{1-t} a_j(1)^t)$, then $\lambda_1({\mathbf a}(t))$
satisfies the inequality
$$\lambda_1({\mathbf a}(t))
\le (1-t)\lambda_1({\mathbf a}(0) + t\lambda_1({\mathbf a}(1))
\ \ {\it  for} \ \  0<t<1 . $$
\end{theorem}

As mentioned above, when $\beta=0$. the $\mu({a_j})=\pi/(2 a_j)$.
Once again the theorem is trivial to prove when $\beta=0$ as then
\begin{equation}\lambda_1(\beta=0)
= \frac{\pi^2}{4} \sum_{j=1}^N \exp(-2\ell_j) .
\label{eq:lambda1beta0lj}
\end{equation}
The proof when $\beta>0$ is given in~\S\ref{subsec:thm3MIA}.
Amongst the several different ways to prove Theorem~\ref{thm:thm3ArXivV2} is one in which two of
the ingredients are (i) Theorem~\ref{thm:thm3MIA} and (ii) that
taking logs transforms the equal volume constraint from a product one involving the $a_j$
to a sum involving the $\ell_j$.

Some of the behaviour as one scales the domain is given in the following:

\begin{theorem}\label{thm:thmScale}
(i) Let $\Omega_1$ be a~\boxshape.
Define $\Omega_t=t\Omega_1$.
Then $\lambda_1(\Omega_t)$ is  decreasing in $t$, convex in $t$ and
further, AG-convex and HA-convex.\\
(ii) For $\beta=0$, for any $\Omega_1$, not merely for  \boxshape{es}
$ \lambda_1(\Omega_t)$ is completely monotone.
\end{theorem}
Item (i) follows from the corresponding properties for $\mu_{(2)}$ given in Lemma~\ref{lem:rlem1}.
By taking $\Omega_0=\{0\}$ one may see some similar results, concerning $\mu$ rather than  $\mu_{(2)}$.\\
Item (ii) is trivial as it is merely the statement that $1/t^2$ is completely monotone on $t>0$.
We emphasiset, this statement concerning changes of scale is, when $\beta=0$ true for any domain $\Omega$.
We have included it here 
to assist in exposition in~\S\ref{sec:Discussion}.

When $\beta=0$ there are many further results for convex domains.
Consider the Minkowski sum, defining a $\Omega_t=(1-t)\Omega_0+t\Omega_1$.
Then, when $\beta=0$,  $\lambda_1(\Omega_t)^{-1/2}$ is a concave function of $t$.
Equation (5) of ~\cite{Col05} gives
\begin{equation}
\lambda_1(\Omega_t)^{-1/2}
\ge (1-t)\lambda_1(\Omega_0)^{-1/2} + t \lambda_1(\Omega_1)^{-1/2}\qquad
{\rm when\ \ }\ \beta=0 .
\label{in:BrunnMin}
\end{equation}
While we do not know if there is any result, when $\beta>0$ for general convex domains,
for \boxshape{es} and $\beta\ge{0}$ we have a result.
Our \boxshape{es} centred on the origin behave nicely under Minkowski sums.
The set of all  our \boxshape{es} centred at the origin is closed under Minkowski sums:
$$\Omega(t):= \Omega((1-t){\mathbf a} + t{\mathbf b})= (1-t)\Omega({\mathbf a})+t\Omega({\mathbf b}) . $$
(That when $\beta=0$,  $\lambda_1(\Omega_t)^{-1/2}$ is a concave function of $t$ is easily checked for our \boxshape{es},
When $\beta=0$ using formula~(\ref{eq:lambda1beta0}), one verifies
inequality~(\ref{in:BrunnMin}) by using the $r=-2$ form of the Minkowski inequality
 as given in~\cite{HLP34}, \S2.11, p30.)
If the property does extend to general convex domains and $\beta>0$, some of the ingredients of the $\beta=0$
proof, e.g.
homogeneity of domain functionals, are lost when we fix $\beta>0$.

\begin{theorem}\label{thm:thmBorell}
Use the notation above for the Minkowski sum of \boxshape{es} and let the fundamental eigenvalue $\lambda_1(\Omega(t))$
be as given in~\S\ref{sec:explicit} equation~(\ref{eq:lambda1Rect}).
For all $\beta\ge{0}$, $\lambda_1(\Omega(t))^{-1/2}$ a concave function of $t$. That is
$$ \frac{1}{\sqrt{\lambda_1(\Omega(t))}}
\ge
 \frac{1-t}{\sqrt{\lambda_1(\Omega(0))}}+  \frac{t}{\sqrt{\lambda_1(\Omega(1))}}
 \ \ {\it  for} \ \  0\le t\le 1 .
$$
\end{theorem}

When $N=2$, Corollary~\ref{cor:corPerim} to Theorem~\ref{thm:thm3ArXivV2} follows from
Theorem~\ref{thm:thmBorell} and the observation that
if the rectangles $\Omega_0$ and $\Omega_1$, $\Omega_1$ being obtained by rotating $\Omega_0$ through
a right angle, then the
perimeter of $\Omega_t$ is constant in $t$.
Symmetry suggests that the maximum of the concave function $\lambda_1(\Omega(t))^{-1/2}$
will occur at $t=1/2$, i.e. the square.

For the next theorem, though results in $N$ dimensions are available, for ease of exposition, results presented here
and  in  subsection~\S\ref{subsec:LS} are,
unless otherwise indicated, for $N=2$.
Part (ii) of the theorem is not proved in this paper: see~\cite{LS11}.
The theorem presents  interesting results involving, as well as $\lambda_1(\Omega)$,
the polar moment of inertia about the centroid, $I_c(\Omega)$.
For our rectangles $(-r h,rh)\times{(-h/r,h/r)}$,the  area is $4h^2$, and
$$ I_c(r) = \frac{4}{3} h^4\left( r^2 +\frac{1}{r^2}\right)
\qquad{\rm while}\qquad
\lambda_1(\beta=0)=\frac{\pi^2}{4 h^2} \left( r^2 +\frac{1}{r^2}\right) . $$

\begin{theorem}{\label{thm:thmLS}}
Define
$${\cal R}(\Omega):= \frac{\lambda_1(\Omega) |\Omega|^3}{I_c(\Omega)} .$$
(i)  ${\cal R}(\Omega)$
is maximal among rectangles for the square.\\
(ii)  ${\cal R}(\Omega)$
is maximal among
triangles for the equilateral triangle, maximal among parallelograms for the square, and
maximal among ellipses for the disk.
\end{theorem}

If we choose to consider families of shapes with the same volume, the results
involve the ratio ${\lambda_1(\Omega)}/{I_c(\Omega)}$.

This theorem (and more) is proved in~\cite{LS11}.
Our inequality~(\ref{eq:VxVymuSqSimple}) establishes the $\beta>0$ version of Theorem~\ref{thm:thmLS}(i):
see~\S\ref{subsec:LS}.

Some of the history is given in~\cite{La11}.
In~\cite{PoS51} it is noted that, when $\beta=0$,  ${\cal R}=12\pi^2$ is constant for rectangles,
which, however, renders the $\beta=0$ case of Theorem~\ref{thm:thmLS}(i) somewhat trivial.
Tables of ${\cal R}(\Omega)$ for various plane shapes , and $\beta=0$, are given in~\cite{PoS51} p257.
Parallelograms are considered in~\cite{Po52}.
Also Hersch~\cite{He66}, equation (5),  establishes the result
for $\beta=0$ and $N=2$, and states it as
{\it Of all parallelograms having given distances between their parallel sides, the rectangle has the
largest $\lambda_1$.}
A variational proof of the result, with $N=2$ and $\beta=0$ is straightforward.

There are some subtleties when $\beta$ is nonzero.
We will not be proving (ii) in this paper: see~\cite{LS11}.

\subsection{Steiner symmetrization applied when $\beta=0$}\label{subsec:Steiner}

Let $\Omega$ be a bounded simply connected domain and $u$ a non-negative function vanishing on
the boundary $\partial\Omega$.
Steiner symmetrization (defined in~\cite{PoS51}) was introduced by Steiner in order to prove the
classical geometric inequality
$$C_\odot(N) |\Omega|^{N-1}\le |\partial\Omega|^N
\qquad{\rm with\ \ }
C_\odot = N^{N-1} \frac{2\pi}{\Gamma(N/2)}  . $$
Here $C_\odot$ is the value of ${|\partial\Omega|^N}/{|\Omega|^{N-1}}$ taken when $\Omega$ is a ball, and the result can be approached
by various symmetrizations,
If Steiner symmetrization is used, one needs many succeessive Steiner symmetrizations
(about planes through the centroid) and to take limits.
In most physical applications $N=2$ or $N=3$,
Steiner symmetrization can also be used, with care, for some polygonal or polytope domains.
Steiner symmetrization changes a polygon into a polygon, but
typically increases the number of sides.  (For a use, see~\S\ref{subsec:famousPS}.)

The classical reference concerning Steiner symmetrization defined for functions and it uses
is the book~\cite{PoS51}.
One of the topics there  concerns the fundamental Dirichlet eigenvalue, i.e. $\beta=0$.
Let $v$ be any (sufficiently smooth) non-negative function defined on $\Omega$
and vanishing on the boundary.
Denote Steiner symmetrizations with a superscript $s$.
Steiner symmetrization preserves the $L_2$ norm of $v$:
$$ \int_{\Omega^s} (v^{(s)})^2 = \int_\Omega v^2 . $$
Steiner symmetrization reduces the Dirichlet `energy':
$$\int_{\Omega^s} |{\rm grad}(v^{(s)})|^2
\le \int_\Omega |{\rm grad}(v)|^2 . $$
The consequences of this for the fundamental Dirichlet eigenvalue date back nearly a century,
to papers by Faber and Krein, but with~\cite{PoS51} as a classical account.
Define
$${\cal V}_0(\Omega,v):= \frac{\int_\Omega |{\rm grad}(v)|^2|}{\int_\Omega v^2} . $$
Then, with $u$ the fundamental eigenfunction for $\Omega$,
$$\lambda_1(\Omega^s)
= \min {\cal V}_0(\Omega^s,v)
\le{\cal V}_0(\Omega^s,u^{(s)})
\le {\cal V}_0(\Omega,u)
=\lambda_1(\Omega) . $$
Repeated Steiner symmetrization can be used to establish, over domains $\Omega$
with the same volume, that the ball has the least fundamental eigenvalue.

Symmetrization techniques are not directly applicable when $\beta>0$, and in this paper we
restrict our study to rectangular parallelepipieds, here called {\boxshape}es.

\subsection{A famous question of Polya and Szego~\cite{PoS51}}\label{subsec:famousPS}

There is a famous question of Polya and Szego~\cite{PoS51}, p159.
{\it Amongst all $n$-gons of given area, does the regular $n$-gon have the smallest $\lambda_1$?}
Symmetrization techniques are used, when $\beta=0$ and $N=2$ in~\cite{PoS51} pp158-159
to establish this to the case when $n=3$ and $n=4$ (but the answer is unknown for $n>4$).
Our Theorem~\ref{thm:thm3ArXivV2} concerns $\beta>0$ and, at least is not inconsistent with the
possibility that, as is the case with $\beta=0$,
the square is the optimizer over all quadrilaterals with the given area.

Consider next $N\ge{2}$.  Symmetrization can be applied when $\beta=0$ and $N\ge{2}$ considering all
parallelepipeds: see~\cite{PoS51} p159 item (d).
Item (d) states that for any prism (right or oblique) with a given volume and a quadrilateral base,
there is a set of successive Steiner symmetrizations which will transform the prism to a cube.

\section{The explicit formulae for $\lambda_1$ for a \boxshape}\label{sec:explicit}

The function $u=\prod_{j=1}^N \cos(\mu({a_j}) x_j)$ satisfies
\begin{equation}
 \Delta u + \lambda u = 0 ,\qquad {\rm with}\
\lambda = \sum_{j=1}^N \mu({a_j})^2  .
\label{eq:lambda1Rect}
\end{equation}
(Here we have chosen to look for modes which are symmetric about the axes, which is appropriate for the fundamental mode.
In other applications, this need not be appropriate.
The general setting is given in~\cite{GN13}\S3.1.)
The Robin boundary conditions are satisfied if, for all $j$,
\begin{equation}
\mu({a_j})\tan(a_j \mu({a_j}) )=\frac{1}{\beta} .
 \label{eq:muXmuY}
 \end{equation}

The function $\mu(c)$ and the geometry determine $\lambda_1$:
Another organization of the equation is
  \begin{equation}
  \mu(c) \tan( c\mu(c))=\frac{1}{\beta},\qquad
  {\mbox{\rm or equivalently}}\qquad
{\hat\mu} \tan({\hat c}{\hat\mu})=1,
 \qquad{\mbox{\rm where\ \ }} {\hat\mu}=\beta\mu,\ {\hat c}=\frac{c}{\beta} \  .
 \label{eq:rrectmuchat}
 \end{equation}
 The transcendental equations have been widely studied, e.g.~\cite{BS73,MRS03,LWH15}.
 Numerical values, often used for checks, are given in Table 4.20 of~\cite{AS65}.
Amongst the various applications are (i) the energy spectrum for the
one-dimensional quantum mechanical finite square well,
and (though with $c<0$)
(ii) (though with $c<0$) zeros of the spherical Bessel function $y_1(x)=j_{-2}(x)$.

 We have an interest in the smallest positive solutions,
\begin{displaymath}0<\mu({a_j})<\pi/(2 a_j),\qquad   0<{\hat \mu}<\pi/(2 {\hat c}) .
\end{displaymath}

\section{The one-dimensional problem}{\label{sec:Oned}}

Determining properties of $\mu(c)$ directly involves much implicit differentiation and is somewhat ugly that way.
However $\mu$ has a simple inverse function $\phi_1$, defined in equation~(\ref{eq:rphiarctan}):
$$\frac{c}{\beta} = \phi_1(\beta\mu) . $$
It is overwhelmingly easier to obtain properties of the inverse $\phi_1$ (and of $\phi_2(z)=\phi_1(\sqrt{z})$)
and from these deduce properties of $\mu$ using results from~\citeKeconv.
The main results relevant to this paper, and proved in~\citeKeconv are given in the next two Lemmas.
The diagrams in~\S{5} of~\citeKeconv help summarise the information given in the next Lemma.

\begin{lemma}
\label{lem:rlem1}{\it Let
\begin{equation}
\phi_1(z)=\frac{1}{z} \, \arctan\left( \frac{1}{z} \right) ,\qquad
\phi_2(z)=\frac{1}{\sqrt{z}} \, \arctan\left( \frac{1}{\sqrt{z}} \right)  .
\label{eq:rphiarctan}
\end{equation}
Then both $\phi_1(z)$ and $\phi_2(z)$ are positive, monotonic decreasing, convex functions on $0<z<\infty$.
Both $\phi_1$ and $\phi_2$ are  log-convex (AG-convex), and, even stronger, both are completely monotone.
\begin{description}
  \item[] $\phi_2$ is a Stieltjes function so AH-convex (as well as AG-convex) while HA-concave.
  \item[] $\phi_1$ is HA-convex, AG-convex and AH-concave.
  \item[] The inverse of $\phi_1$ is $\mu$ and  $1/\mu$ is concave (i.e. $\mu$ is AH-convex and hence, also, $\log(\mu)$ is convex,
$\mu$ is AG-convex).  $\mu$ is aso HA-concave.
  \item[] The inverse of $\phi_2$ is $\mu_{(2)}$ which is   log-convex -- AG-convex, and
HA-convex (hence also GA-convex, i.e.
is such that $\mu_{(2)}(\exp(\ell))$ is convex in $\ell$).
$\mu_{(2)}$ is also AH-concave.
\end{description}
}
\end{lemma}


The preceding lemma is required for the main results of this paper.
Definitions are given systematically in ~\citeKeconv, but sometimes, as immediately below, repeated in this paper.
Further results follow from additional properties of $\mu$, e.g.
Theorem~\ref{thm:GG} follows from properties in Lemma~\ref{lem:rlem1GG}.

{\sc Definition}\
A function $f:(0,\infty)\rightarrow{(0,\infty)}$ is said to be {\it GG-convex} iff
$$ f(\sqrt{x_0 \, x_1}) \le \sqrt{f(x_0) \, f(x_1)}
 .$$
With the inequality reversed it is {\it GG-concave}.

\medskip

\begin{lemma}
\label{lem:rlem1GG}{\it  The functions $\phi_2$ and $\mu_{(2)}$
(and $\phi_1$ and $\mu$) are GG-concave.}
\end{lemma}

\section{$\lambda_1$ for {\boxshape}es with $|\Omega|$ fixed}\label{sec:VolFixed}

The principal topic of this section is proofs, plural, of Theorem~\ref{thm:thm3ArXivV2}.
We begin by noting that
\begin{equation}
g \ {\mbox{\rm is GA-convex\ }}\
\Longleftrightarrow
g\left( \left(\prod_{j=1}^N a_j\right)^{1/N}\right)
\le \frac{1}{N}\sum_{j=1}^N g(a_j) .
\label{eq:gGA}
\end{equation}
As noted in Lemma~\ref{lem:rlem1},
that $\mu_{(2)}$ is GA-convex is a consequence of the log-concavity, AG-concavity, of
the Stieltjes function $\phi_2$.
The first proof begins from noting that this theorem, as re-stated in~\S\ref{subsec:Overview}, is merely that
$\mu_{(2)}$ is GA-convex, and replacing the $g$ of the preceding equation by  $\mu_{(2)}$
is the result.

\medskip

An earlier paper of the authors,~\cite{KW16s}, involves isoperimetric results
for a physical problem for which only $N=2$ is relevant.
It was appropriate in this earlier paper to provide a self-contained, short
proof of the $N=2$ version of Theorem~\ref{thm:thm3ArXivV2}.
The proof we give in~\S\ref{subsec:thm3} is an adaptation of the proof in the earlier paper,
and, essentially, is another proof that for a positive, decreasing, convex function $f$ which is
AG-convex, its inverse $g$ is GA-convex.
The proof in ~\S\ref{subsec:thm3} uses Lemma~\ref{lem:lem2} from the next subsection.

\subsection{Separable convex optimization problems}\label{subsec:SeparableC}

In the paragraph preceding Theorem~\ref{thm:thm3MIA} in \S\ref{sec:rIntro} we mentioned that
properties of $\mu$ beyond that of it being log-convex are required in order to prove Theorem~\ref{thm:thm3ArXivV2}.
Here is an example illustrating that.

{\sc Example.} For the problem of minimizing $\lambda(a_1,a_2)=\exp(-a_1)+\exp(-a_2)$ subject to
the positive numbers $a_1$, $a_2$ satisfying $a_1 a_2=h^2>0$, the minimum need not occur on the
line $a_1=h=a_2$.

By writing the constraint as $a_1=r h$ and $a_2=h/r$, we note that the objective function can be studied as
a function of the single variable $r$, and is $L(r,h)=\lambda(r h,h/r)=\exp(-r h)+\exp(-h/r)$.
The critical point of $L(.,h)$ at $r=1$ is  a local minimum if $h>1$ and a local maximum if  $h<1$.
As $L(r,h)\rightarrow{1}$ as $\rightarrow{0}$ and $L(1,h)>1$ for $h<\log(2)$,
we see that, as asserted,
the minimum does not occur on the line  $a_1=h=a_2$.

\begin{lemma}
\label{lem:lem2}{\it Let $\zeta:(0,\infty)\rightarrow{\mathbf R}$ be twice continuously differentiable,
decreasing and convex.
Then the separable convex programming problem, given $A>0$,
\begin{displaymath}{\it to\ find\ \ } \{a_j | 1\le{j}\le{N} \} \qquad
{\mbox{\it which minimize}} \ \ \sum_{j=1}^N a_j
\end{displaymath}
and satisfy the constraint
\begin{equation} \sum_{j=1}^N \zeta(a_j)=  N\zeta(A) ,
\label{nu:zetaDef}
\end{equation}
is solved by $a_j=A$ for all $j$.}
\end{lemma}

\begin{proof}
The convexity of $\zeta$ gives (as in~\cite{HLP34}\S3.6 p71)
\begin{equation} \zeta\left(\frac{\sum_{j=1}^N a_j}{N}\right)
\le \frac{\sum_{j=1}^N \zeta(a_j)}{N}=\zeta(A) .
\label{nu:zetaCvx}
\end{equation}
(This is a particular case of the finite, or discrete, form of Jensen's inequality.)
But $\zeta$ is decreasing so
\begin{equation}
\frac{\sum_{j=1}^N a_j}{N} \ge A , \qquad \sum_{j=1}^N a_j \ge N\, A .
\label{nu:XYZend}
\end{equation}
The lower bound of $NA$ on the sum $\sum_{j=1}^N a_j$ is attained when
$a_j=A$ for all $j$.\hfill\qed
\end{proof}

The hypotheses in the following lemma are greater than needed, but we have written it in a way
which invites its application building from Theorem~\ref{thm:thm3MIA}.

\begin{lemma}\label{lem:lemMSepConv}
{\it Let $M:{\mathbf R}\rightarrow(0,\infty)$ be twice continuously differentiable,
decreasing and convex.
Then the separable convex programming problem, given a real number $L$,
\begin{displaymath}{\it to\ find\ \ } \{\ell_j | 1\le{j}\le{N} \} \qquad
{\mbox{\it which minimize}} \ \ \sum_{j=1}^N M(\ell_j)
\end{displaymath}
and satisfy the constraint
\begin{equation} \sum_{j=1}^N \ell_j =  N L ,
\label{nu:Mconstraint}
\end{equation}
is solved by $\ell_j=L$ for all $j$.
Using strictness in the decreasing and convex, the minimizer is unique.}
\end{lemma}

Once again the proof depends on Jensen's inequality.
We omit the details.

\begin{remark}\label{rem:expl}
Suppose functions $M$ and $m$ are related as $M(\ell)=m(\exp(\ell))$
or equivalently, for $z>0$, as $m(z)=M(\log(z))$.
Then if $M$ is decreasing and convex, so is $m$.
\end{remark}
Differentiating gives
$$\frac{d m}{d z} = \frac{1}{z} M'(\log(z)) ,$$
and
$$\frac{d^2 m}{d z^2} =- \frac{1}{z^2} M'(\log(z))+\frac{1}{z^2} M''(\log(z)) .$$
The result in the remark follows.
The remark can also be expressed: if $m$ is AG=convex, it is also AA-convex.

The converse of Remark~\ref{rem:expl} is false as, for example, if $m(z)=\exp(-z)$ we have that $m$ is convex
but $M(\ell)=m(\exp(\ell))=\exp(-\exp(\ell))$ is not convex.

\subsection{Another proof of the isoperimetric inequality of Theorem~\ref{thm:thm3ArXivV2}}\label{subsec:thm3}

In the approach here all $\mu({a_j})$ are specified, and these determine $h$ and $\mu=\mu_\square$.

\begin{proof} (of Theorem~\ref{thm:thm3ArXivV2}).\
In this second proof, we begin by noting
that, at fixed $\beta$ and $h$, $\mu(r h)$ is positive and monotonic decreasing in $r$.
Let us now suppose we are given $\beta$, and all the $\mu({a_j})$.
Write
\begin{equation}
\beta\mu({a_j}) = {\hat\mu}_{a_j},\qquad \beta \mu_\square = {\hat\mu}_\square .
\label{nu:beta3Def}
\end{equation}
The goal is to show
\begin{equation}
 \sum_{j=1}^N \mu({a_j})^2 \ge N \mu_\square^2 = \lambda_{1\square} ,
\label{nu:nuGoal}
\end{equation}
or equivalently
\begin{equation}
 \sum_{j=1}^N {\hat\mu}_{a_j}^2 \ge N {\hat\mu}_\square^2 .
\label{eq:goalB}
\end{equation}
Now, in the notation of equation~(\ref{eq:rphiarctan})
\begin{equation}
\frac{a_j}{\beta} = \phi_1({\hat\mu}_{a_j})= \phi_2({\hat\mu}_{a_j}^2) ,\quad
\frac{h}{\beta} = \phi_1({\hat\mu}_\square) = \phi_2({\hat\mu}_\square^2)  .
\label{nu:phihr}
\end{equation}
Now, with $\beta$ given,  $h$ is determined by the values of $\mu({a_j})$:
\begin{equation}
 \left(\frac{h}{\beta}\right)^N
= \prod_{j=1}^N \phi_1({\hat\mu}_{a_j})= \prod_{j=1}^N \phi_2({\hat\mu}_{a_j}^2).
\label{nu:phi2Prod}
\end{equation}
Also
\begin{equation} \left(\frac{h}{\beta}\right)^N
= \phi_1({\hat\mu}_\square)^N = \phi_2({\hat\mu}_\square^2)^N.
\label{nu:phisquare}
\end{equation}
Eliminating $h/\beta$ between the two preceding equations we have
\begin{equation}
\prod_{j=1}^N \phi_1({\hat\mu}_{a_j}))= \phi_1({\hat\mu}_\square )^N ,\qquad
\prod_{j=1}^N \phi_2({\hat\mu}_{a_j}^2)= \phi_2({\hat\mu}_\square^2 )^N .
 \label{nu:phiElim}
\end{equation}

 We now apply Lemma~\ref{lem:lem2} with $\zeta(z)=\log(\phi_2(z))$ and note that Lemma~\ref{lem:rlem1}
 ensures that $\phi_2$ is decreasing and  log-convex so that
 the conditions on $\zeta$, needed for Lemma~\ref{lem:lem2} are satisfied.
 This establishes the result.\hfill\qed
 \end{proof}

The stronger inequality, $\sum_{j=1}^N{\hat\mu}_{a_j} >  N{\hat\mu}_\square$ ,  follows from considering $\phi_1$:
see~\cite{KW16s}.

\subsection{Simplifications and further results when $N=2$ and $|\Omega|$ fixed}\label{subsec:FurtherN2}

Define, for $r>0$, $\Omega(rh,h/r)=(-r h,r h)\times(-h/r,h/r)$. Write $\lambda_1(rh,h/r)=\lambda_1(\Omega(rh,h/r))$.
When the orientaion of the rectangle is unimportant, abbreviate these to $\Omega(r)$ and $\lambda_1(r)$.
Clearly $\lambda_1(1/r)=\lambda_1(r)$.

Theorem~\ref{thm:thm3ArXivV2} when $N=2$, that  $\mu_{(2)}$ is GA-convex, is
that, for $r>0$,
$$ \mu_{(2)}(h)= \mu_{(2)}(\sqrt{ rh\frac{h}{r}})
\le \frac{\mu_{(2)}(rh)+\mu_{(2)}(\frac{h}{r})}{2} .
$$

Theorem~\ref{thm:thm3MIA}, another result from $\mu_{(2)}$ being GA-convex, gives that
$\lambda_1$ is a convex function of $\log(r)$, and, clearly, an even function of $\log(r)$.
So, again, we see that the minimum of $\lambda_1(r)$ is at $r=1$ as stated in Theorem~\ref{thm:thm3ArXivV2}.

The slightly stronger result that $2\mu(h)\le{\mu(rh)+\mu(h/r)}$ follows from $\phi_1$
being AG convex, so $\mu$ is GA-convex.
This is
the $N=2$, the inequality $\sum_{j=1}^N{\hat\mu}_{a_j} >  N{\hat\mu}_\square$ from the end of the
preceding subsection: it trivially rewrites to
$$\lambda_1\ge \lambda_{1\square} +\frac{1}{2}\left(\mu(rh)-\mu(h/r)\right)^2
 .
$$

\medskip
HA-convexity is a stronger result than GA-convexity, and yields the following,
a result that was originally conjectured from the plots of $\lambda_1(r)$ given in~\cite{KW16s}.

\begin{theorem}\label{thm:HA}
$\lambda_1(r)$ is a convex function of $r$.
\end{theorem}

\begin{proof}
The positive, decreasing, convex function $\mu_{(2)}$ is HA-convex, so $\mu_{(2)}(1/c)$ is convex in $c$.
In an obvious notation $\lambda_1(r)=\mu_{(2)}(r)+\mu_{(2)}(1/r)$, and, as both functions on the right
are convex, $\lambda_1(r)$ is convex.
\end{proof}

We remark, but do not use, that the convex function $\mu$ is HA-concave, so the function
$\mu(r)-\mu(1/r)$ is convex.

\medskip

Additional results can be found from further properties of $\phi$ and/or of $\mu$,
e.g. $\mu$ -- and $\mu_{(2)}$ --  are $GG$-concave.

Suppose now that the rectangle $\Omega(rh,h/r)$ has $0<r\le{1}$.
We now need to consider the orientation of the rectangle, so
now denote its fundamental Robin eigenvalue by $\lambda_1(rh,h/r)$.
Consider the largest square which can be inscribed in the rectangle
(the square  $\Omega(hr,hr)$and
the smallest square in which the rectangle can be inscribed
(the square  $\Omega(h/r,h/r)$.
The isoperimetric inequality is that  $\lambda_1(h,h)$ is bounded above by
   $\lambda_1 (hr,h/r)$ which is the arithmetic mean of $\lambda_1 (rh,rh)$ and $\lambda_1(h/r,h/r)$.
(This is our old result Theorem~\ref{thm:thm3ArXivV2},)

The  GG-concavity of ($\mu$ and) $\mu_{(2)}$  yields the following:

\begin{theorem}\label{thm:GG}
$\lambda_1(h,h)$ is bounded from below by
  the geometric mean of $\lambda_1 (rh,rh)$ and $\lambda_1(h/r,h/r)$.
\end{theorem}

The key point in the proof is that each of the three terms concerns a $\lambda_1$ for a square,
and $\lambda_1(c,c)=2\mu_{(2)}$.
Thus the theorem  is merely a restatement of the GG-concavity of $\mu_{(2)}$.

We also state this in terms of the $a_j$ which is suggestive of how it might generalize to $N>2$:
$$\sqrt{\lambda_1(a_1,a_1)\lambda_1(a_2,a_2)}
\le \lambda_1(\sqrt{a_1 a_2},\sqrt{a_1 a_2})
\le \frac{1}{2}\left(\lambda_1(a_1,a_1) + \lambda_1(a_2,a_2) \right) .
$$
When $\beta=0$, the left-hand inequality becomes an equality.

\section{Proofs of Theorems stated in \S\ref{subsec:Othergeom}}\label{sec:stated1p1}

\subsection{Proof of Theorem~\ref{thm:thm3MIA}, convexity in $\log(a_j)$ }\label{subsec:thm3MIA}

\begin{proof} of Theorem~\ref{thm:thm3MIA}.
This follows from the GA-convexity of $\mu$ or of $\mu_{(2)}$. 
\end{proof}

The GA-convexity of $\mu_{(2)}$ establishes both Theorems~\ref{thm:thm3ArXivV2} and~\ref{thm:thm3MIA}.
Furthermore Theorem~\ref{thm:thm3MIA} implies the result of
Theorem~\ref{thm:thm3ArXivV2}.
We illustrate this with $N=2$.
With the earlier notation where $r$ is related to the
aspect ratio of the rectangle (with $r=1$ for the square),
$$\lambda_1=\mu_{(2)}(r) +\mu_{(2)}(\frac{1}{r})
= M(\log(r))+M(-\log(r)) . $$
As each of the $M$ terms is convex, $\lambda_1$ is a convex function of $\log(r)$.
Also  $\lambda_1$ is an even function of $\log(r)$.
Hence the minimum of  $\lambda_1$ occurs at $r=1$, $\log(r)=0$.
For larger $N$,
Lemma~\ref{lem:lemMSepConv} and Theorem~\ref{thm:thm3MIA},
together, yield yet another proof of Theorem~\ref{thm:thm3ArXivV2}.

\medskip
Stronger results follow from using the HA-convexity of $\mu_{(2)}$.
When $N=2$ this is indicated in~\S\ref{subsec:FurtherN2}.

\subsection{Proof of Theorem~\ref{thm:thmScale}, scaling}\label{subsec:Scale}

Part (i) is merely a collection of some of the properties of $\mu_{(2)}$ which are closed under addition.
Part (ii) is a consequence of $\lambda_1(\Omega_t)=\lambda_1(\Omega_1)/t^2.$

\subsection{Proof of Theorem~\ref{thm:thmBorell}, Minkowski sums}\label{subsec:Minkowski}

Consider the Minkowski sum
$$ \Omega_t = (1-t)\Omega_0 + t\Omega_1 . $$
As stated in \S\ref{sec:rIntro}, when $\beta=0$ it is known that $\lambda_1(\Omega_t)^{-1/2}$ is a concave function of $t$.

The result  at $N=1$, that $1/\mu$ is concave,
$\mu$ is AH-convex,  is given in~Lemma~\ref{lem:rlem1}. 

Results more general than the following lemma are proved in~\cite{Li82}.

\begin{lemma}\label{lem:lemBorell1}
 Let the real-valued functions $f$ and $g$ have the same domain of definition $D\subset{R^N}$
with (i) $D$ convex, (ii) $f$ and $g$ positive and twice continuously differentiable in $D$, and
(iii), with $-1<\alpha<0$ both $f^\alpha$ and $g^\alpha$  concave.
Then $(f+g)^\alpha$ is concave.
\end{lemma}

\begin{proof}
Denote the column vector gradient with a $D$, and the hessian by $D^2$,
and transpose with a superscript $T$.
Define, for any $f$ which is positive and $C^2$,
$$M(f,\alpha)= f\,  D^2 f +(\alpha-1)\,  Df\,  (Df)^{T}  .$$
Note that
$${\rm hessian}(f^\alpha) =  \alpha\, f^{\alpha-2}\, M(f,\alpha) . $$
To establish the result we need to prove that $M(f+g,\alpha)$ is positive (semi-)definite,
given $M(f,\alpha)$ and $M(g,\alpha)$ both positive semidefinite.
\begin{align*}
M(f+g,\alpha)
&= (f+g)\, D^2 (f+g) + (\alpha-1) D(f+g)\,  (D(f+g))^{T}  ,\\
&= (f+g)\left( \frac{1}{f} M(f,\alpha)+\frac{1}{g} M(g,\alpha)\right) +(1-\alpha) W W^T ,
\end{align*}
with $W=\sqrt{g/f}\, Df -\sqrt{f/g}\, Dg$.
The first two terms of the preceding equation are positive (semi-)definite by the hypotheses
of the lemma, and the last is positive semi-definite.
Hence $M(f+g,\alpha)$  is positive (semi-)definite, completing the proof.
\end{proof}

\begin{proof} (of Theorem~\ref{thm:thmBorell}).
Since $\mu_{(2)}(a)$ is positive, decreasing with $1/\sqrt{\mu_{(2)}(a)}$ concave, we have,
by Lemma~\ref{lem:lemBorell1} with $\alpha=-1/2$,  that
$\Lambda({\mathbf a})=\sum_{j=1}^N \mu_{(2)}(a_j) $ is  such that
$1/\sqrt{\Lambda({\mathbf a})}$ is concave in the (convex) positive orthant,
$\{ {\mathbf a} | a_j>0\}$.
\end{proof}


\subsection{Variational methods}\label{subsec:variational}

It is possible to establish inequalities leading to another proof,
via a variational approach, of Theorem~\ref{thm:thm3ArXivV2}.
As we note below, the calculations are detailed.
Variational methods also yield other results including Theorem~\ref{thm:thmLS}(i).

Let $W^1_2(\Omega)=H_0^1$ be the Sobolev space of functions with square integrable first derivatives.
For domains $\Omega$ with sufficiently smooth boundaries (including our rectangles),
the fundamental eigenvalue is given by
$$\lambda_1
= {{\rm min}\atop{v\in{W^1_2(\Omega)}}}{{\cal V}_\beta}(\Omega,v)\qquad{\rm where\ } \ \
{{\cal V}_\beta}(\Omega,v)= \frac{G(v) + E(v)/\beta}{F(v)} 
$$
and the gradient term $G$, and the terms $F$ and $V$, are
$$ G(v)  = \int_\Omega |\nabla v|^2,
\qquad F(v)  = \int_\Omega v^2,
\qquad E(v)  = \int_{\partial\Omega} v^2.
$$
While the methods are equally applicable in $N$ variables, here we report on the calculations presented in
detail in~\cite{KW16s}, which are for $N=2$ with $x_1$ denoted by $x$ and $x_2$ by $y$.
We used test functions of the form
$$ v = \cos(\nu_X x) \cos(\nu_Y y) ,
$$
for various choices of $\nu$.

With
$\nu_X=\nu_Y=\sqrt{\lambda_{1\square}/2}$ and $\Omega(r)$ a rectangle with $a_1=r h$ and $a_2= h/r$, we find
\begin{equation}
\lambda_1(\Omega(r))\le{\cal V}_\beta(\Omega(r),v)
=\lambda_\square\left(
 \frac{1}{2}\left(
r^2+\frac{1}{r^2}\right) -
\frac{\beta(\sqrt{r}-\frac{1}{\sqrt{r}})^2 (r+1+\frac{1}{r})}
{h(1+\frac{\beta^2\lambda_\square}{2}) +\beta}
\right) .
\label{eq:VxVymuSq}
\end{equation}
This implies the weaker inequality
 \begin{equation}
\lambda_1(\Omega(r))
\le
\frac{\lambda_\square}{2}\left(
r^2+\frac{1}{r^2}\right) ,
\label{eq:VxVymuSqSimple}
\end{equation}
and we remark there is equality in this when $\beta=0$.
\goodbreak

With an appropriate choice of $\nu$ and applying the variational principle over the square
$\Omega_\square$, one can find rather elaborate inequalities which imply the inequality of
Theorem~\ref{thm:thm3ArXivV2}.
In this, as in the derivation of inequality~(\ref{eq:VxVymuSqSimple}),
we use the exact solutions of \S\ref{sec:explicit}.
See~\cite{KW16s}.
The calculations are more tedious than the simple proof we gave in~\S\ref{subsec:thm3}.

\subsection{Results of Laugesen and Siudeja}\label{subsec:LS}


Theorem~\ref{thm:thmLS}(i) is easier to prove when $\beta=0$.
It isgeneralized to Robin boundary conditions (and sums of consecutive eigenvalues) in~\cite{LS11},
and further generalized to higher dimensions in~\cite{LS12}.
We remark that the rectangle version of the result
is the same as our inequality~(\ref{eq:VxVymuSqSimple})
(and we note that we have not considered parallelograms in our variational calculations for $\beta>0$).

Further results are presented in~\cite{FS10,LS15}.

\section{Conclusion --  and open questions}\label{sec:Discussion}

\subsection{The eigenvalue problem}\label{subsec:DiscussEig}

We have, in Theorem~\ref{thm:thm3ArXivV2},  established an isoperimetic inequality for the fundamental Robin eigenvalue
for {\boxshape}es. We have also reviewed some related inequalities.
In the process we found that working with the explictly defined Stieltjes function $\phi_2$ is overwhelmingly neater
than the elementary, but detailed, calculations that arise when working with  $\mu_{(2)}$, the inverse of $\phi$.
Many other inequalities on $\lambda_1$, and related functionals,
for domains more general than {\boxshape}es, have been proved and others have been conjectured:
see, for example, ~\cite{FS10,LS15} and the long arXiv article~\cite{KW16s}.


In~\cite{PoS51} it is asked if,
{\it amongst all $n$-gons of given area, that which has the least $\lambda_1$ is the regular $n$-gon.}
When $\beta=0$ this is, in~\cite{PoS51} page 158, proved to be the case,
using symmetrisation, when $n=3$ and $n=4$.
See earlier in this paper, at the end of \S\ref{subsec:Steiner}.
For $\beta>0$, the question for triangles, $n=3$, is noted as {\it Open Problem 1} in~\cite{LS15}.
Assuming this is found to be true, it would make sense to ask which classes of quadrilaterals have the property that, at given area, the square has the least $\lambda_1$.
Our little result, Theorem~\ref{thm:thm3ArXivV2},  is that it is true for rectangles,
 but it is open as to whether it is true for some larger class, e.g. parallelograms or trapeziums.

Returning to {\boxshape}es we note some open questions.
There are indications, mentioned in~\cite{KW16s} that ${\hat\mu}({\hat c})$ might be completely monotone,
and the first question is, is it, and, if so, how might it be used.
We merely used the log-convexity of $\phi_2$ in our proof of Theorem~\ref{thm:thm3ArXivV2},
and various other convexity properties in other theorems.
However we have the stronger property that $\phi_2$ is a Stieltjes function, so completely monotone, and it is reasonable to ask
what further properties of $\lambda_1(\Omega)$ can be obtained from this.
We have already commented (in Theorem~\ref{thm:thmScale}(ii)) that when $\beta=0$,
under the scale change,  for any $\Omega_1$.
$\lambda_1(t\Omega_1)$ is completely monotone.
It is an open question as to whether this is true when $\beta>0$, even for $\Omega_1$ a rectangle.

It is an open problem to find a very neat  proof of Theorem~\ref{thm:thm3ArXivV2} by variational techniques,
those reported in~\S\ref{subsec:variational} involving lengthy, but routine, calculation, even for $N=2$.
\goodbreak

\subsection{The generalized torsion problem}\label{subsec:DiscussTors}

In the applications in~\cite{MK94,KM93} the problem  P($\beta$) with $\torsforce>0$ and $\lambda=0$,
mentioned in~\S\ref{subsec:Overview}, also arises.
The physically interesting cases are $N=2$ and $N=3$.
A quantity of interest is generalized torsional rigidity, denoted by $S_\beta$ in these papers, and with $N=2$ by
$Q_{\rm steady}$ in~\cite{KW16s}.
The generalization of the St Venant Inequality for the case $\beta>0$ is proved in~\cite{BG15}.
When $N=2$ and $\beta>0$ there is good numerical evidence, reported in~\cite{KW16s},
 suggesting that a result corresponding to our
Theorem~\ref{thm:thm3ArXivV2} should be available for $S_\beta=Q_{\rm steady}$.
See~\cite{PoS51} for the proof for $N=2$ and $\beta=0$.
When $\beta>0$ there is, as yet, only numerical evidence, and no proved result.

When $\beta=0$ results analogous to Theorem~\ref{thm:thmBorell} are available.
See~\cite{Bo85}.
When $\beta=0$ and $N=2$, $S_\beta^{1/4}(t)$ is concave.
We do not know how this might extend to $\beta>0$, even for rectangles --
and there even for simple scale-changes.

\bigskip
\begin{center}
{\Large Supplement to Part II}
\end{center}


\section*{Formerly in~\S\ref{sec:rIntro}}


\medskip
Before considering the case when $\beta>0$, we note that the plane case with $\beta=0$ is reported in~\cite{Po52}.
Tables of ${\cal R}(\Omega)$ for various plane shapes , and $\beta=0$, are given in the
earlier publication~\cite{PoS51} p257.
In particular ${\cal R}(\Omega)$ is the same constant for any rectangle.
Polya's proof is variational, using ${\cal V}_0$ defined in~\S\ref{subsec:Steiner}.
Let $T$ be the linear map, taking the origin to the origin,
 from the square $\Omega_\square=(-h,h)\times{(-h,h)}$ to a parallelogram which,
without loss of generality, may be taken to have one pair of sides parallel to the $x$-axis.
We consider $T$ with $\det(T)=1$, and in calculation use $T$ with matrix representation
$$T =\left[
\begin{array}{cc}
1& m\\
0& 1
\end{array}
\right]
$$
The boundary of the parallelogram is $T(\partial{\Omega_\square})$.
Now
$$ u_\square = \cos(\frac{\pi x}{2 h} ) \cos(\frac{\pi y}{2 h}) $$
is the fundamental  Dirichlet eigenfunction for the square.
With $v$ as the composition $v=u\circ{T^{-1}}$, the function $v$ vanishes on the boundary of the parallelogram.
Routine calculation gives
$$\frac{\lambda_1(T(\Omega))}{I_c(T(\Omega))}
\le
\frac{{\cal V}_0(T(\Omega),v)}{I_c(T(\Omega))}
= \frac{\lambda_{1\square}}{I_c(\Omega_\square)} . $$

Related maple code is presented at the URL given in the abstract.

\subsection*{Eigenfunctions for the Laplacian}\label{subsec:eigLaplace}

Eigenfunction expansions are widely used in linear problems involving Laplacians in the space variables and homogeneous boundary conditions: see~\cite{CHI57}.
\begin{itemize}
\item
Applied to the homogeneous heat equation, the behaviour at large time is dominated by the fundamental eigenvalue.
(In the authors' application to flows in microchannels, in which $N=2$, the explicit infinite series for starting flows
as given in the arXiv supplement form of~\cite{KW16s} is an example.
\item
Eigenfunction expansions are also appropriate for the homogeneous undamped wave equation.
Here, without damping, higher eigenfunctions remain as important as the fundamental.
However if there is any damping the higher eigenfunctions are damped more than the fundamental.
Hence, again, the fundamental dominates at large time.
\item
Eigenfunction expansions can also be used to the non-homogeneous Poisson equation.
One expands the forcing in terms of the eigenfunctions and then determines the coefficients in the expansion for the solution.
\end{itemize}

\section*{Remarks on \S\ref{subsec:SeparableC}}

\begin{proof} of~\ref{lem:lemMSepConv} 
The non-negativity of $M$ ensures that an infimum exists.
As with Lemma~\ref{lem:lem2}, we apply Jensen's inequality.
This gives, from the convexity of $M$,
$$\frac{1}{N}\sum_{j=1}^N M(\ell_j) \ge M(\frac{\sum_{j=1}^N \ell_j}{N}) . $$
Using the constraint~(\ref{nu:Mconstraint}) this is
$$\sum_{j=1}^N M(\ell_j) = N\ M(L) , $$
which is clearly solved by $\ell_j=L$ for all $j$.

Obviously if $M$ were to be a constant function, for example, there would be many solutions to
the optimization problem.
The additional conditions stated give the uniqueness.
\qed
\end{proof}

\section*{Remarks on  \S\ref{subsec:FurtherN2}}

Code relevant to \S\ref{subsec:FurtherN2} is presented at the
URL given in the abstract.


\section*{Remarks on \S\ref{subsec:Minkowski}}

My own proof of Lemma~\ref{lem:lemBorell1} requires $f$ and $g$ to be $C^2$.
Stronger results are in the Lindberg~\cite{Li82} reference. 

\medskip
{The next two lemmas are old bits editted out of the final paper, as neater proofs
were found.}

\begin{lemma}\label{lem:lemBorellc}
If $\mu_2$ is convex and $1/\sqrt{\mu_2}$ is concave on $(0,\infty)$ then so also is the sum
$c+\mu_2$ for any nonnegative constant $c$.
\end{lemma}

\begin{proof}
Since ${\mu_2} {\mu_2}'' -\frac{3}{2}({\mu_2}')^2\ge{0}$ and ${\mu_2}''\ge{0}$ one sees that the same
inequalities are satisfied when, with $c\ge{0}$, one replaces ${\mu_2}$ by $c+{\mu_2}$.
\end{proof}

Lemma~\ref{lem:lemBorellc} doesn't get used, but is in the draft as it may indicate that concavity in the coordinate
directions isn't affected by the $\mu$ contributions of the other coordinates.

My first efforts to remove the restriction to $N\le{2}$ in Lemma~\ref{lem:lemBorell},
before noting the result of~\cite{Li82}, were clumsy.
I first checked that it is OK (with $\mu_{(2)}$ for $N=3$ but brute force calculation isn't going to get general $N$.
Some form of induction might.

\begin{lemma}\label{lem:lemBorell}
Let $N\le{2}$. If $\mu_2$ is positive, decreasing, convex and logconvex and
$1/\sqrt{\mu_2}$ is concave on $(0,\infty)$ then 
$$\sigma({\mathbf a})
= \frac{1}{\sqrt{\sum_{j=1}^N \mu_2(a_j)}} $$
is concave on the positive orthant.
\end{lemma}

\begin{proof}
The result is trivially true when $N=1$. Let $N=2$ and set
\begin{equation}
 {\mu_2(a_1)}'' -\frac{3({\mu_2(a_1)}')^2}{2 \mu_2(a_1)}= A_1\ge{0}, \qquad
 {\mu_2(a_2)}'' -\frac{3({\mu_2(a_2)}')^2}{2 \mu_2(a_2)}= A_2\ge{0} .
\label{eq:Adef}
\end{equation}
To align the calculations with the future application, also define
$$\Lambda({\mathbf a})
=\sum_{j=1}^N \mu_2(a_j) .$$
We are required to establish that the hessian of $\sigma({\mathbf a})$ is negative definite.

Denote the column vector gradient with a $D$, and the hessian by $D^2$,
and transpose with a superscript T.
Define
$$M(\Lambda,n)= -n\,  D\Lambda\,  (D\Lambda)^{T} +\Lambda\,  D^2 \Lambda .$$
Note that
$${\rm hessian}(1/\sqrt{\Lambda}) =  - \frac{1}{2}\Lambda^{5/2}\, M(\Lambda,\frac{3}{2}) . $$
To establish the result we need to prove that $M(\Lambda,3/2)$ is positive (semi-)definite.
Now the second term $D^2\Lambda$ is a diagonal matrix with positive entries.
However $D\Lambda\,  (D\Lambda)^{T}$ is positive semidefinite so that
for $n>0$ this term works against establishing that $M$ is positive definite.
The logconvexity of $\mu_2$ gives $M(\Lambda,1)$ positive definite,
i.e. $\Lambda$ is logconvex.
However, the stronger result on $\sigma$ concerning $M(\Lambda,3/2)$ has eluded us at general $N$,
so we turn now to the calculation at $N=2$,

When $N=2$, abbreviating $M(\mu_2(a_1)+ \mu_2(a_2),n)$ to $M(n)$,
$$M(n)
=
-n\left[ \begin{array}{cc}
(\mu_2'(a_1))^2& \mu_2'(a_1) \, \mu_2'(a_2) \\
\mu_2'(a_1)\, \mu_2'(a_2)& (\mu_2'(a_2))^2
\end{array}\right] + (\mu_2(a_1)+ \mu_2(a_2)\, 
\left[ \begin{array}{cc}
\mu_2''(a_1)&0 \\
0& \mu_2''(a_2)
\end{array}\right]
$$
On eliminating the second derivatives using equation~(\ref{eq:Adef}), we find
$$
M(\frac{3}{2})
= \left[ \begin {array}{cc} 
\,{\frac {2\, (\mu_2 \left( a_1 \right)+ \mu_2\left( a_2 \right) ) A_1+
3\, \mu_2\left( a_2 \right)  \left( {\frac {d}{da_1}}\mu_2 \left( a_1 \right)  \right)^{2}}
 {2 \mu_2\left( a_1\right) }}&
 -\frac{3}{2}\, \left( {\frac {d}{da_1}} \mu_2\left( a_1 \right)  \right) {\frac {d}{da_2}} \mu_2\left( a_2 \right) \\ 
 \noalign{\medskip}
 -\frac{3}{2}\, \left( {\frac {d}{da_1}}{\mu_2} \left( a_1 \right)  \right) {\frac {d}{da_2}}{\it \mu_2}\left( a_2 \right) &
 \,{\frac{ {2\,( \mu_2\left( a_1\right) +\mu_2} \left( a_2 \right) ) A_2+
 3\, \mu_2 \left( a_1 \right)  \left( {\frac {d }{da_2}} \mu_2\left( a_2 \right)  \right)^{2}}
{2  \mu_2\left( a_2 \right) }}
\end {array} \right] 
 . $$
 The diagonal elements are positive and the determinant is also positive, so $M(3/2)$ is positive definite, and the
 result is established.
\end{proof}

\medskip

\noindent{\bf Informal comments}

For rectangles/boxes it may be that some properties are inherited from the
Dirichlet case.
The function
$$ u = \prod_{j=1}^N \cos(\mu_j x_j) $$
along with solving the Robin b.c. problem H($\beta$)
satisfies zero Dirichlet b.c.s on a larger rectangle/box,
with
$$A_j =\frac{\pi}{2\mu_j} . $$
In particular, the fundamental eigenfunction is log-concave in its rectangle.\\
Of course, from the explicit solution, this is very easy to see.
With $f(x)=\cos(\mu x)$, we have $f f'' - (f')^2=-\mu^2$.
More simply $f$ positive and concave implies $f$ is log-concave.
Products of log-concave functions are log-concave.

Calculations (as yet unchecked)

\section*{Remarks on \S\ref{subsec:variational}}

Maple code for this topic is presented via the URL given in the abstract.




\section*{A note on how we discovered the proofs}

The colloboration of the present authors  began when GK noticed that earlier work of his,
joint with Alex McNabb, on heat flow problems, would lead to results for
Benchawan's flows with slip at the boundary.
When checking developments since the work with McNabb, GK noticed that there were
recent isoperimetric inequalities, by Daner and others around 1990 for the
fundamental Robin eigenvalue and from 2015 for the generalization of the St Venant inequality.
Both of these concerned circular disks being the optimizers over cross-sections with the same area.
A short note to an engineering journal which published papers on microchannels
reporting this was rejected with the referee saying ``we know this", presumably on the
basis of numerical work. 

Benchewan had a paper concerning the series solutions for rectangular microchannels.
This gave numerical indications that square cross-sections would be the
optimizers over all rectangles with the same area.
The challenge GK took up was to prove this, at least for the fundamental Robin eigenvalue.
The rest of this note concerns the fundamental eigenvalue.

We discovered the results in the following order.

\begin{itemize}

\item First we established, by direct calculation of formulae for derivatives,
that $\mu(c)$ was convex, even log-convex (AG-convex).
This, however, did not lead to a proof of the isoperimetric theorem.
Higher derivatives of $\mu(c)$ were calculated, not with any application in mind,
but leading to the conjecture that perhaps $\mu(c)$ might be completely monotone.

\item Because the formulae for $\phi_1$ and $\phi_2$ were so simple, it was
immediately noted that they were completely monotone, and shortly after, that
$\phi_2$ was Stieltjes.
To begin with we didn't see that this would lead to a proof of the isoperimetrci
inequality, but in time, the log-convexity of $\phi_2$ yielded its proof.
(We now know 
(i) that the isoperimetric inequality is just the GA-convexity of $\mu_{(2)}$, and
(ii) our proof was really one showing that the inverse of a GA-convex function
is an AG-convex function.)\\
A revised microchannel paper was prepared and is under review.

\item It was noted that the fundamental eigenvalue results generalized to $N$-dimensions.
A short (and we think well written) paper was submitted
(possibly to a `wrong' journal), but was not accepted.

\item Inequalities other than the isoperimetric ones were available for Dirichlet
boundary conditions, $\beta=0$, e.g. that $1/\sqrt{\lambda_1}$ is concave under
Minkowski convex combinations.
The initial proofs of the extension of this -- and other theorems -- to $\beta\ge{0}$
required various convexity properties of $\mu$.
Going back and forth between considering convexity properties of $\phi$ and of $\mu$
first led to Fact 1 of~\citeKeconv.
Eventually it led to considering $(p,q)$-convexity and Theorem 1 of~\citeKeconv.
The generality in~\citeKeconv was very much more than needed for the Robin eigenvalue
application so the paper was split into the two Parts as indicated in this report.

\end{itemize}

\end{document}